\documentclass[11pt, reqno]{amsart}
	\usepackage{amssymb, latexsym, amsmath, amsfonts, mathrsfs, yhmath}
	\usepackage{graphicx}
	\usepackage{bm}
	\usepackage{float}
	\usepackage{tikz}
        \usepackage{tikz-cd} 
	\usepackage{enumitem}
	\usepackage[numbers]{natbib}
	\usepackage{hyperref}
	\usepackage{upgreek}
	\makeatletter
\@namedef{subjclassname@2020}{%
  \textup{2020} Mathematics Subject Classification}
\makeatother
	\newtheorem{thm}{Theorem}[section]
	\newtheorem{cor}[thm]{Corollary}
	\newtheorem{lem}[thm]{Lemma}
	
	\newtheorem{prop}[thm]{Proposition}
	\theoremstyle{definition}
	
	\theoremstyle{remark}
	\newtheorem{rem}[thm]{Remark}
	\numberwithin{equation}{section}
	
	\numberwithin{equation}{section}	
	
    \providecommand{\bysame}{\leavevmode\hbox to3em{\hrulefill}\thinspace}
	\newcommand{\C}{\mathbb{C}}
	\newcommand{\ra}{\rightarrow}
	\newcommand{\deck}{\upgamma}
	\newcommand{\al}{\mathfrak{a}}
	\newcommand{\z}{\zeta}
	\newcommand{\ov}{\overline}
	\newcommand{\sm}{\setminus}
	\newcommand{\ep}{\epsilon}
	\newcommand{\no}{\noindent}
	\newcommand{\Om}{\Omega}
	\newcommand{\ti}{\tilde}
        \newcommand{\w}{\widetilde}
	\newcommand{\iso}{\mathfrak{i}}
	
	\newcommand{\Aut}[1]{{\rm Aut}(#1)}
	\newcommand{\I}[1]{\textsf{int}\left(#1\right)}
        \newcommand{\h}{\widehat}
	\newcommand{\pr}{\mathbf{p}}
    \newcommand{\G}{\mathfrak{G}}
        \newcommand{\cover}{\mathbb{C} \times \left(\mathbb{C} \setminus \overline{\mathbb{D}}\right)}
    \newcommand{\p}{\mathrm{P}}
	
\begin{document}
	\title{Rigidity of the escaping set of polynomial automorphisms of $\C^2$}
	\keywords{H\'{e}non map, Escaping set, Automorphism group, Short $\C^2$}
	\subjclass[2020]{Primary: 37F80, 32M18; Secondary: 32Q02, 32T05}
	\author{Sayani Bera, Kaushal Verma} 
\address{SB: School of Mathematical and Computational Sciences, Indian Association for the Cultivation of Science, Kolkata-700032, India}
\email{sayanibera2016@gmail.com, mcssb2@iacs.res.in}

\address{KV: Department of Mathematics, Indian Institute of Science, Bengaluru-560012, India}
\email{kverma@iisc.ac.in}

\begin{abstract}
Let $H$ be a polynomial automorphism of $\C^2$ of positive entropy and degree $d \ge 2$. We prove that the escaping set $U^+$ (or equivalently, the non-escaping set $K^+$), of $H$ is rigid under the action of holomorphic automorphisms of $\C^2$. Specifically, every holomorphic automorphism of $\C^2$ that preserves $U^+$ takes the form $L \circ H^s$ where $s \in \mathbb{Z}$ and $L$ belongs to a finite cyclic group of affine maps that preserve the escaping set.

\medskip Second, note that the sub--level sets $\{G^+ < c\}$, $c > 0$, of the Green’s function $G^+$ associated with the map $H$ are canonical examples of {\it Short} $\C^2$'s. As a  consequence of the above theorem, we show that the holomorphic automorphisms of these {\it Short} $\C^2$'s are affine automorphisms of $\C^2$ preserving the escaping set $U^+$. Hence, the automorphism group of these {\it Short} $\C^2$'s are the same for every $c>0$ and is a finite cyclic group.
\end{abstract}
\maketitle

\section{Introduction} 
The goal of this article is to study a {\it rigidity} property of the escaping set (or equivalently, the non-escaping set) of polynomial automorphisms of $\C^2$ of positive topological entropy. By a classical theorem of Friedland--Milnor \cite{FM}, they are conjugate to a polynomial automorphism $H$ of the form
\begin{align}\label{e:henon}
H(x,y)=H_m \circ \cdots \circ H_1(x,y)
\end{align}
where $H_i(x,y)=(y, p_i(y)-a_i x)$, $p_i$ is a polynomial in one variable with degree at least $2$ and $a_i \neq 0$ for every $1 \le i \le m$. Such maps $H$ and their factors $H_i$ as in (\ref{e:henon}) will be referred to as {\it generalised} and {\it simple} H\'enon maps respectively. Note that the degree of the first and second coordinates of $H$ are  
\[
d'=d_{m-1}\hdots d_1 \text{ and } d=d_md_{m-1}\hdots d_1
\]
respectively, and that the Jacobian determinant of $H$ is a constant and equals $a = a_1a_2 \ldots a_m$.

\medskip 

Let ${\rm Aut}(\C^2)$ denote the group of holomorphic automorphisms of $\C^2$. We identify all elements of ${\rm Aut}(\C^2)$ that preserve the escaping set (or equivalently, the non-escaping set) associated with a generalised H\'enon map $H$. Before proceeding, let us first recall the following definitions. The {\it non-escaping set} or the {\it bounded Julia set} associated with a generalised H\'{e}non map $H$ is defined as
$K^+=\{z \in \C^2: \|H^n(z)\| \text{ is bounded}\}$ and the {\it escaping set} is $U^+=\C^2 \setminus K^+$. Furthermore, $H$ extends to $\mathbb{P}^2$ as a birational map with an indeterminacy at the point $I^+=[1:0:0]$ and a super-attracting fixed point at $I^-=[0:1:0]$. The escaping set $U^+$ corresponds to the intersection of the basin of attraction of $I^-=[0:1:0]$ with $\C^2$, and $I^+$ is the only accumulation point of the non-escaping set $K^+$ in $\mathbb{P}^2$. See \cite{BS2}, \cite{FS} for these details. The inverse of $H$,  
\[H^{-1}(x,y)=H_1^{-1} \circ \cdots \circ H_m^{-1}(x,y)\]
has an indeterminacy at $I^-=[0:1:0]$ when extended to $\mathbb{P}^2$ and its {\it non-escaping} and {\it escaping sets} are defined as $K^-=\{z \in \C^2: \|H^{-n}(z)\| \text{ is bounded}\}\text{ and }U^-=\C^2 \setminus K^-$ respectively.
Furthermore, $U^-$ corresponds to the basin of attraction of $I^+=[1:0:0]$ of $H^{-1}$ when extended to $\mathbb{P}^2$ and $\ov{K^-}=K^- \cup I^-$.  

\medskip

\medskip The phenomenon of {\it rigidity} for generalised and simple maps H\'enon has various manifestations. To place this work in context, it is useful to recall a few related results.

\medskip

First, given a pair of generalised H\'enon maps $H$ and $F$ that have the same non-escaping set, Lamy \cite{L:alternative} showed that $H, F$ must be the same up to possibly different iterates. That is, there exist positive integers $m, n$ such that $H^m = F^n$. As a consequence, $H^{-1}$ and $F^{-1}$ will also have the same non-escaping and escaping sets.

\medskip

Second, the notion of {\it very rigid} appears in the work of Dinh--Sibony \cite{DS:rigidity} who showed that there exists a unique positive $dd^c$-closed $(1,1)$ current of mass $1$ on $K^+$. This extends results from \cite{BS2}, \cite{FS}, and leverages the potential-theoretic properties of the Green's function associated with $K^+$.

\medskip

Third, \cite{BPV:rigidity} shows that if $f \in {\rm Aut}(\C^2)$ preserves both $K^{\pm}$, then the following dichotomy holds: either $f$ is linear or one of $f$ or $f^{-1}$ must be a generalised H\'enon map that commutes with $H$ up to a diagonal unimodular change of variables. As a consequence, the non-escaping sets of $f$ (or $f^{-1}$ as the case may be) and $H$ must be the same. The proof uses the very rigid property of $K^\pm$ alluded to above and the dependence of the Green's functions on the  B\"{o}ttcher functions near the super-attracting fixed points $I^\pm$ in $\mathbb{P}^2$ constructed in \cite{Hu-Ov1}. Furthermore, by combining the above results, it is shown in \cite{BPV:IMRN} that if $f$ is an affine map that preserves the non-escaping set of a {\it simple} H\'{e}non map $H$, then $f$ is conjugate to the linear map $(x, y) \mapsto (\eta x, \eta^d y)$ where $\eta$ is a $(d^2-1)$-root of unity, and $f \circ H^2= H^2\circ f$. 

\medskip

Next, Pal in \cite{P:biholomorphic} showed that two simple H\'{e}non maps $H, F$ with the same degree that have biholomorphic escaping sets are related via affine automorphisms $A_1$ and $A_2$, i.e., $H=A_1 \circ F \circ A_2$. This generalises an earlier result of Bonnot--Radu--Tanase \cite{BRT:biholomorphic}.

\medskip

Finally, and more recently, Cantat--Dujardin \cite{CD:rigidity} show that if a pair of polynomial automorphisms $f, g$ in ${\rm Aut}(\C^2)$ with dynamical degree $d \ge 2$ are conjugate by some $\varphi \in {\rm Aut}(\C^2)$, then $\varphi$ must be a polynomial automorphism. This confirms a conjecture of Friedland--Milnor and can be viewed as a rigidity result for generalised H\'enon maps. Related to this circle of ideas on conjugacy and motivated by a question of Favre \cite{F:Kato_surfaces}, Bacher \cite{Ba:Kato} shows that a pair of generalised H\'enon maps are conjugate in ${\rm Aut}(\C^2)$ precisely when the Kato surfaces associated to them near the super-attracting fixed point $I^{-}$ are biholomorphic.

\medskip

The results in \cite{BPV:IMRN} and \cite{P:biholomorphic} build on the unpublished work of Bousch \cite{Bousch} that uses the description of an intermediate cover of $U^+$ given by Hubbard--Oberste-Vorth in \cite{Hu-Ov1} to compute the automorphism group of $U^+$ for simple quadratic H\'{e}non maps.  Our goal is to further generalise these results by identifying all automorphisms of $\C^2$ that preserve $U^+$ associated with generalised H\'enon maps. Here is the main result:

\begin{thm}\label{t:main}
Let $H$ be a generalised H\'enon map of the form (\ref{e:henon}), i.e., 
\begin{align*}
H(x,y)=H_m \circ \cdots \circ H_1(x,y)
\end{align*}
where $H_i(x,y)=(y, p_i(y)-a_i x)$, $p_i$ is a polynomial with degree $d_i \ge 2$ and $a_i \neq 0$ for every $1 \le i \le m$. Then the non-escaping set $K^+$, or equivalently, the escaping set $U^+$, associated with $H$ is rigid under the action of ${\rm Aut}(\C^2)$. In other words, any $f \in {\rm Aut}(\C^2)$ that satisfies 
$$f(K^+) = K^+\text{ or equivalently } f(U^+)=U^+,$$ 
is of the form $f^r=L \circ H^s\text{ with }r,s \in \mathbb{Z}$ and $r>0$, where $L$ is an affine map that preserves both $K^\pm$. In addition,
\[\mathcal{L}=\{L: L \text{ is an affine map such that }L(K^\pm)=K^\pm\} \le \mathbb{Z}_{(d+d')(d-1)} \]
\end{thm}

\noindent Here, $f$ is essentially a rational root of the H\'{e}non map $H$, in the sense of \cite{BF:roots}. However, the assumption that $f(K^-) = K^-$, which was imposed in \cite{BPV:rigidity} is not required, and this is the main thrust of this theorem. Further, in the spirit of \cite[Theorem 1.1]{BPV:rigidity} it follows that there exists a H\'{e}non map $\rho_H$ (an irreducible root of $H$, again in the sense of \cite{BF:roots}) such that
\[  f=L \circ \rho_H^s \text{ and }f \circ H=\tilde{L} \circ H \circ f\]
where $L, \tilde{L} \in \mathcal{L}$ and $s \in \mathbb{Z}$. The proof consists of several steps that are summarized as follows. 

\medskip

First, Theorem \ref{t:HOv-gen} in Section $2$ extends the construction of the intermediate cover of $U^+$ by Hubbard--Oberste-Vorth in \cite{Hu-Ov1} to the case of generalised H\'enon maps. Next, let ${\rm Aut}_1(U^+)$ be the subgroup of the holomorphic automorphism group ${\rm Aut}(U^+)$ whose elements induce the identity map on the fundamental group $\pi_1(U^+)$. Theorem \ref{t:HOv-gen} leads to a better understanding of the structure of ${\rm Aut}(U^+)$. This is the content of Theorem \ref{t:escaping set}, which shows that modulo a copy of $\C$, ${\rm Aut}(U^+)$ is at most finite cyclic. The proof relies on adapting several ideas from \cite{Bousch}. Using this information about the structure of ${\rm Aut}(U^+)$, Theorem \ref{t:Aut(U+) cap Aut(C2)} identifies all elements of ${\rm Aut}(U^+) \cap {\rm Aut}(\C^2)$ whose lifts to the intermediate cover of $U^+$ constructed above have a specific form. The final step is completed in Theorem \ref{t:linear maps} which extends the observations made in the proof of Theorem \ref{t:Aut(U+) cap Aut(C2)} and shows that ${\rm Aut}(U^+) \cap {\rm Aut}(\C^2)$ is a finite cyclic group whose order can be bounded above in terms of the degree of $H$. Similar results were observed earlier in \cite{B:rigidity-II}.

\medskip
Recall that the Green's functions associated to the sets $K^\pm$ are defined as
\[G_H^+(z)=\lim_{n \to \infty} \frac{1}{d^n} \log^+\|H^n(z)\|\text{ and }G_H^-(z)=\lim_{n \to \infty} \frac{1}{d^n} \log^+\|H^{-n}(z)\|.\]
They are non-negative, continuous, plurisubharmonic functions on $\C^2$ and satisfy the functorial property 
\begin{align}\label{e:functorial} 
G_H^\pm(H^s(z))=d^{\pm s}G_H^\pm(z)
\end{align}
for every $s \in \mathbb{Z}$.
Further, $K^\pm=\{G_H^\pm=0\}$ and 
\[ \mu_H^+=\frac{1}{2 \pi} dd^c G_H^+\text{ and }\mu_H^-=\frac{1}{2 \pi} dd^c G_H^-\]
are $dd^c$-closed positive $(1,1)$ currents of mass 1 supported on $K_H^+$ and $K_H^-$, respectively. As indicated earlier, $\mu_H^\pm$ are precisely those currents that bestow on $K_H^{\pm}$ their very rigid property. 

\medskip

\noindent For every $c>0$, the sub-level set
\[\Omega_c=\{z \in \C^2: G_H^+(z)<c\}\]
is known to be a {\it Short} $\C^2$. {\it Short} $\C^2$'s are domains whose infinitesimal Kobayashi metric vanishes identically but at the same time admit a non-constant bounded plurisubharmonic function and which can be exhausted by an increasing union of domains biholomorphic to the ball. They were first observed and constructed by Forn{\ae}ss in \cite{Fornaess:ShortC2}. 

\medskip 

In \cite{BPV:IMRN} the analytical and topological structure of an intermediate cover of a punctured {\it Short} $\C^2$, i.e., $\Omega_c'=\Omega_c\setminus K^+$, corresponding to a simple H\'enon map was explored in a manner similar to \cite{Hu-Ov1}, \cite{BRT:biholomorphic} and \cite{Bousch}.  The philosophy behind this analysis is to understand the rigidity of an appropriate punctured neighbourhood of the super-attracting fixed point $I^-$ of $H$, in contrast to the rigidity of the Kato surfaces associated to the neighbourhood of $I^-$, raised by Favre in \cite{F:Kato_surfaces}. 

\medskip As a consequence, it is proved therein that the holomorphic automorphism group of certain punctured {\it Short} $\C^2$'s corresponding to a simple H\'{e}non map can at most be finitely many copies of $\C$. Let ${\rm Aut}(\Omega_c)$ and ${\rm Aut}(\Omega'_c)$ denote the holomorphic automorphism groups of $\Omega_c$ and $\Omega'_c$ respectively. The subgroup of ${\rm Aut}(\Omega'_c)$ consisting of those elements that induce the identity map on the fundamental group $\uppi_1(\Omega'_c)$ will be denoted by ${\rm Aut}_1(\Omega'_c)$.

\medskip Theorem \ref{t:HOv-gen} has consequences for understanding both ${\rm Aut}(\Omega_c)$ and ${\rm Aut}_1(\Omega'_c)$. First, we observe in Theorem \ref{t:revised BPV} that covers of punctured {\it Short} $\C^2$'s obtained as sub-level sets of generalised H\'enon maps can be constructed. This leads to  Corollary \ref{c:punctured short} in which the automorphism group of such punctured {\it Short} $\C^2$'s can be identified in the same spirit as \cite[Theorem 1.2]{BPV:IMRN}. Furthermore, we will use this representation to conclude the following:


\begin{thm}\label{t:equivalence of automorphisms}
Let $H$ be a generalised H\'enon map. Then
\[{\rm Aut}_1(U^+)={\rm Aut}_1(\Omega_c')\] 
for every $c>0$.
\end{thm}

\noindent By combining Theorem \ref{t:main} and Theorem \ref{t:equivalence of automorphisms}, we obtain  
\begin{cor}\label{c:automorphism of C2}
Let $H$ be a generalised H\'enon map and $f \in {\rm Aut}(\Omega_c)$ for some $c>0$. Then there exists an affine map $L \in {\rm Aut}(\C^2)$ such that $L(K^\pm)=K^\pm$ and $L(z)=f(z)$ on 
$\Omega_c$.
\end{cor}

\noindent In other words,

\begin{cor}
Let $H$ be a generalised H\'enon map. Then for every $c>0$, every element of ${\rm Aut}(\Omega_c)$ is the restriction of an affine map that preserves the escaping set (or non-escaping set) of $H$ and ${\rm Aut}(\Omega_c)$ is a finite cyclic group of order at most $\mathbb{Z}_{d_0(d-1)}$. Here, $d_0$ is an integer that depends on $d+d'$.
\end{cor}

\subsection*{Acknowledgements:} The first author is partially funded by the Mathematical Research Impact Centric Support (MTR/2023/001435) grant from the Anusandhan National Research Foundation (ANRF) of India.

\section{An intermediate cover of the escaping set \texorpdfstring{$U^+$}{U+}}

The goal of this section is to prove an analog of the Hubbard--Oberste-Vorth result on covering spaces of $U^+$ in \cite{Hu-Ov1} for generalized H\'{e}non maps of the form (\ref{e:henon}). Here is the precise statement.
\begin{thm}\label{t:HOv-gen}
Associated with the subgroup $\mathbb{Z}$ of $\uppi_1(U^+)$, there exists an open cover $\widehat{U}^+$ of $U^+$ that is biholomorphic to $\cover$. The generalised H\'enon map $H$ admits a lift to $\widehat{U}^+$ and the lift is of the form
\begin{align}\label{e:HOv-Q}
\w{H}(z, \z)=\left(\frac{a}{d}z +Q(\z), \z^d\right),
\end{align}
where \[
Q(\z) = \z^{d+d'} + A_{d+d'-1} \z^{d+d'-1}  + \ldots+A_0. 
\]	
Further, 
\begin{enumerate}
	\item [(i)] The fundamental group $\uppi_1(U^+)$ is isomorphic to $\mathbb{Z}\left[\frac{1}{d}\right]$.
	\item [(ii)] An element $\left[\frac{k}{d^n}\right]$ of $\uppi_1(U^+)/\mathbb{Z}$ corresponds to the following deck transformation of $\cover$: 
\begin{align}\label{e:deck transformations}
\deck_{k/d^n}( z,\z) = \left( z + \frac{d}{a} \sum_{l = 0}^{n-1} \left(\frac{d}{a}\right)^l \left( Q\left(\z^{d^l}\right) - Q \left( \left(e^{2\pi ik/d^n} \cdot \z\right)^{d^l}\right)\right), e^{2 \pi i k/d^n} \cdot \z \right).
\end{align}
\end{enumerate}
\end{thm}
This is a generalization of the results in \cite[Section 8]{Hu-Ov1}. We will briefly revisit the steps, mostly following \cite[Chapter 7]{Ueda:Book}. Note that  
\[H=H_m \circ H_{m-1} \circ \cdots \circ H_1  \text{ with }H_i(x,y)=(y, p_i(y)-a_i x),\]
where $p_i$ are monic polynomials of degree $d_i$, $1 \le i \le m$. Let us first recall from \cite{BS2} and \cite{Hu-Ov1}, a few properties of the dynamics of a H\'{e}non map of the above form at infinity. Consider the filtration of $\C^2$ by 
\[
V_R = \left\{ \vert x \vert, \vert y \vert \le R \right\}, V^{+}_R = \left\{\vert y \vert \ge \max\{\vert x \vert, R\} \right\},\; V^-_R = \left\{ \vert x \vert \ge \max\{\vert y \vert, R\} \right\}. 
\]
Then, there exists a large $R_H > 1$ such that for every $R\ge R_H$ 
\begin{align}\label{e:filtration_1}
H(V^+_R) \subset V^+_R \text{ and }H^{-1}(V^-_R) \subset V^-_R.
\end{align}
Further, 
\begin{align}\label{e:filtration_2}
U^+=\bigcup_{n=0}^\infty H^{-n}(V_R^+) \text { and }U^-=\bigcup_{n=0}^\infty H^{n}(V_R^-).
\end{align}
Thus, $K^\pm \subset \I{V_R \cup V_R^\mp}$ and $\C^2 \setminus \I{V_R \cup V_R^\mp} \subset U^\pm$. The value $R_H>1$ corresponding to the map $H$ (\ref{e:henon}) will be referred to as the {\it radius of filtration} of $H$ and we will write $V^\pm$ in place of $V_{R_H}^\pm$ unless stated otherwise.

\medskip The proof of Theorem \ref{t:HOv-gen} involves several steps and is divided into subsections accordingly.

\subsection{Construction of the cover \texorpdfstring{$\widehat{U}^+$}{}}\label{subsection 1} Let us first recall from \cite{BPV:rigidity} some results on the construction of the B\"{o}ttcher coordinate function for H\'{e}non maps.  Let $(x_n,y_n)=H^n(x,y)$ for $(x,y) \in V^+$. Define 
\[y_{n+1} = y^d_{n}\Big(1 + \left({q(x_n, y_n)}/{y^d_n}\right) \Big) \sim y^d_n\] 
for $(x, y) \in V^+.$ Now by Proposition 2.1 from \cite{BPV:rigidity}
\begin{align*}
\phi(x, y) = y \lim_{n \ra \infty} \prod_{j=0}^{n-1} \left( 1 + \frac{q(x_j, y_j)}{y^d_j}  \right)^{1/d^{j+1}}
\end{align*}
is a well defined holomorphic function on $V^+$ as the limit is uniform on compact subsets of $V^+$. In addition, $\phi \circ H = \phi^d$ and $G_H^+ = \log \vert \phi \vert$ on $V^+$ with
\begin{align}\label{e:Bottcher}
\phi(x,y)=ye^{\gamma(x,y)}\text{ and }\gamma(x,y) \to 0 \text{ as }y \to \infty.
\end{align}
Note that $\omega=\frac{d \phi}{\phi}$ is a closed 1-form in $V^+$ and $H^*\omega=d \omega$ where $d\ge 2$ is the degree of $H$. Hence, as in Proposition 7.3.2 in \cite{Ueda:Book}, $\omega$ extends to a closed 1-form on $U^+$ by (\ref{e:filtration_2}). For a closed curve $C$ in $U^+$ define
\[\alpha(C)=\frac{1}{2 \pi i}\int_C \omega.\]
Also, Proposition 7.3.3 in \cite{Ueda:Book}, which is stated for simple H\'{e}non maps, generalises as:
\begin{prop}\label{p1:ss:1-1}
For a closed curve $C$ in $U^+$, the following holds:
\begin{itemize}
    \item[(i)] $\alpha(H(C))=d \alpha(C)$, where $H(C)$ is the image of the curve $C$ under $H$.
    \item[(ii)] $\alpha(C) \in \mathbb{Z}\left[\frac{1}{d}\right]$ where
    $\displaystyle \mathbb{Z}\left[\frac{1}{d}\right]=\left\{ \frac{k}{d^n}: k,n \in \mathbb{Z}\right\}.$
    \item[(iii)] $\alpha: \displaystyle \uppi_1(U^+) \rightarrow \mathbb{Z}\left[\frac{1}{d}\right]$ is a group isomorphism.
\end{itemize}
\end{prop}
Further, as in Chapter $7$ of \cite{Ueda:Book}, let $\mathcal{H}$ be the subgroup of $\uppi_1(U^+)\simeq\mathbb{Z}[\frac{1}{d}]$ that corresponds to $\mathbb{Z}$, i.e., $\alpha(\mathcal{H})=\mathbb{Z}$. Fix a point $z_0 \in V^+$ and consider all the pairs $(z,C)$ of points $z \in U^+$, where $C$ is a path in $U^+$ with initial point $z_0$ and end point $z$. Define an equivalence relation
\[(z,C)\sim (z',C') \text{ if and only if } z=z' \text{ and } CC'^{-1} \in \mathcal{H}.\]
The covering space $\widehat{U}^+$ is defined as 
\[\widehat{U}^+=\left\{[z,C]: \text{ the set of all equivalence classes  for every }z \in U^+\right\}.\]
The projection map $\widehat{\pi}: \widehat{U}^+ \to U^+$ is defined as $\widehat{\pi}([z,C])=z.$ Equip $\widehat{U}^+$ with the pull-back complex structure so that the projection map $\widehat{\pi}$ becomes holomorphic. Note that $\widehat{\pi}$ is a covering map as it is a local homeomorphism, and the fundamental group of $\widehat{U}^+$ is $\mathbb{Z}$. 

\medskip Consider the holomorphic map $\hat{\phi}$ defined on $\widehat{U}^+$ as
\[\hat{\phi}(p)=\phi(z_0)\exp\left(\int_C \omega\right), \]
where $\phi$ is the B\"{o}ttcher coordinate (\ref{e:Bottcher}) defined on $V_R^+$. Let $\widehat{H}$ be the lift of $H$ to $\widehat{U}^+$ defined as
\[\widehat{H}([z,C]=[H(z),lH(C)],\]
where $l$ is a path in $V^+$ connecting $z_0$ and $H(z_0)$. Now Propositions 7.3.5 and 7.3.6 from \cite{Ueda:Book} collectively generalise as:
\begin{prop}\label{p:UedaBook}
Let $\displaystyle \widehat{V}^+=\{ [z,C] \in \widehat{U}^+: C \text{ is a path in }V^+\}.$
Then
\begin{itemize}
    \item [(i)] for $p \in \widehat{V}^+$, $\hat{\phi}(p)=\phi(\widehat{\pi}(p))$ and $\widehat{\pi}$ is a biholomorphism between $\widehat{V}^+$ and $V^+$.
    \item[(ii)] for $p,p' \in \widehat{U}^+$ if $\hat{\phi}(p)=\hat{\phi}(p')$ and $\widehat{\pi}(p)=\widehat{\pi}(p')$, then $p=p'$.
    \item[(iii)] the functional equation $\hat{\phi}(\widehat{H}(p))=(\hat{\phi}(p))^d$ holds.
\end{itemize}
\end{prop}
The proof is exactly the same as in \cite{Ueda:Book}. Also $\widehat{H}(\widehat{V}^+) \subset \widehat{V}^+.$ The next step is to construct

\subsection{The form of the map \texorpdfstring{$\widehat{H}$}{} on an appropriate subdomain}\label{subsection 2} Recall that $R_H>1$ is the radius of filtration of the map $H$. For every $M>1$ define the sets
\[W_M^+=\left\{(x,y) \in \C^2: |y| > M\max\{|x|,R_H\}\right\}\] 
and
\[\widetilde{W}_M^+=\left\{(x,y) \in \C^2: |\phi(x,y)| > M\max\{|x|,R_H\}\right\}.\]
Note that by (\ref{e:Bottcher}), there exists $M_0>1$ such that for all $M>M_0$, $\widetilde{W}_M^+ \subset V^+$. Hence, consider the map $E_1: \widetilde{W}_M^+ \to W_M^+$ given by 
\[E_1(x,y)=(x, \phi(x,y)).\]
Then $E_1(\widetilde{W}_M^+) =W_M^+$ and $E_1$ is a biholomorphism whose inverse is
\[E_1^{-1}(x,y)=(x, \lambda(x,y))\]
for some holomorphic $\lambda: W_M^+ \to \C$.

\begin{lem}\label{l:preliminary 1}
For $\ep>0$ sufficiently small, there exists $M_\ep>M_0$ such that 
\begin{itemize}
    \item[(i)] for $(x,y) \in \widetilde{W}_{M_\ep}^+$
$$ 1-\ep\le \left| \frac{\phi(x,y)}{y}\right|\le 1+\ep  \text{ and }1-\ep\le  \left|\frac{\partial\phi}{\partial y}(x,y)\right|\le 1+\ep.$$ 
\item[(ii)] for $(x,y) \in W_{M_\ep}^+$
$$ 1-\ep\le \left| \frac{\lambda(x,y)}{y}\right|\le 1+\ep  \text{ and }1-\ep\le  \left|\frac{\partial\lambda}{\partial y}(x,y)\right|\le 1+\ep.$$ 
\end{itemize}
\end{lem}
\begin{proof}
From (\ref{e:Bottcher}) there exists $M_1\ge 0$ such that 
$$ 1-\ep\le \left| \frac{\phi(x,y)}{y}\right|\le 1+\ep  \text{ and }1-\ep\le \left| \frac{\lambda(\tilde{x},\tilde{y})}{\tilde{y}}\right|\le 1+\ep$$ 
for $(x,y) \in \widetilde{W}_{M_1}^+$ and $(\tilde{x},\tilde{y}) \in W_{M_1}^+$ respectively. Let $D^2(z;1)$ denote the polydisk of polyradius $1$ centered at $z$. Now choose $M_\ep> M_1$ such that $D^2(z; 1) \subset \widetilde{W}_{M_1}^+ $ and $D^2(\tilde{z}; 1) \subset W_{M_1}^+ $ for every $z \in \widetilde{W}_{M_\ep}^+$ and $\tilde{z} \in W_{M_\ep}^+$ respectively. By the Cauchy estimates,
$$ \left|\frac{\partial\phi}{\partial y}(x,y)\right|\le 1+\ep \text{ and }   \left|\frac{\partial\lambda}{\partial y}(\ti{x},\ti{y})\right|\le 1+\ep,$$ 
for $(x,y) \in \widetilde{W}_{M_\ep}^+$ and $(\tilde{x},\tilde{y}) \in W_{M_\ep}^+$
respectively. Since $E_1(\w{W}_{M_\ep}^+)=W_{M_\ep}^+$ and $\lambda(x, \phi(x,y))=y$, 
\[\left|\frac{\partial\phi}{\partial y}(x,y)\right|=\left|\frac{\partial\lambda}{\partial y}\left(x,\phi(x,y)\right)\right|^{-1}\ge (1+\ep)^{-1}\ge 1-\ep.\]
A similar argument can be applied to obtain the lower bound on $\left|\frac{\partial\lambda}{\partial y}(\ti{x},\ti{y})\right|$ for $(\tilde{x},\tilde{y}) \in W_{M_\ep}^+$.
\end{proof}

\begin{rem}\label{r:preliminary 2}
By (\ref{e:Bottcher}), $e^{\gamma(x,y)} \to 1$ as $y \to \infty$ on $V^+$. Since
\[ \frac{\partial\phi}{\partial y}(x,y)=e^{\gamma(x,y)}+ \phi(x,y)\gamma_y(x,y)\]
it follows from Lemma \ref{l:preliminary 1} that for $\ep>0$ sufficiently small,
\begin{itemize}
    \item[(i)] $\displaystyle \frac{\phi(x,y)}{y} \to 1  \text{ and }\frac{\partial\phi}{\partial y}(x,y) \to 1$ in $\widetilde{W}_{M_\ep}^+$ as $y \to \infty$,

    \smallskip
\item[(ii)] $\displaystyle \frac{\lambda(x,y)}{y} \to 1  \text{ and }\frac{\partial\lambda}{\partial y}(x,y)\to  1$ in $W_{M_\ep}^+$ as $y \to \infty$.
\end{itemize}
\end{rem}
For the next part of the computations in this subsection, we will fix the constant $M>M_{\ep}$ where $0<\ep<1/2$ is sufficiently small such that $H(\widetilde{W}_M^+) \subset \I{\widetilde{W}_M^+}$. Consider the map $E_2$ on $W_M^+$ defined as
\[E_2(x,y)=(\psi(x,y),y) \text{ where } \psi(x,y)=y \int_0^x \frac{\partial \lambda}{\partial y}(t,y)dt.\]
Note that $\psi$ is well-defined on $W_M^+$ by Lemma \ref{l:preliminary 1}. Next, we claim that $E_2$ is injective on $W_M^+$. To see this, suppose that $E_2(x_1,y_1)=E_2(x_2,y_2)$ on $W_M^+$.  Then $y_1=y_2=y \neq 0$ and
\[y \int_{x_1}^{x_2} \frac{\partial \lambda}{\partial y}(t,y)dt=0, \text{ i.e., }\int_{x_1}^{x_2} \frac{\partial \lambda}{\partial y}(t,y)dt=0,\]
where $x_1,x_2$ belong to the disc $D\left(0,|y|/M\right) \subset \C$. In particular, by Remark \ref{r:preliminary 2} (ii) and the choice of $\ep>0$,
\begin{align}\label{e:lambda}
|x_1-x_2| \le \left|\int_{x_1}^{x_2} \left(\frac{\partial \lambda}{\partial y}(t,y)-1\right)dt\right|\le \frac{1}{2}|x_1-x_2|.
\end{align}
Hence $x_1=x_2$ and this shows the injectivity of $E_2$ on $W_M^+$. Now consider the map $\Psi=E_2 \circ E_1$ on $\widetilde{W}_M^+$. Again, as $\lambda(x, \phi(x,y))=y$ for every $(x,y) \in \widetilde{W}_M^+$,
\begin{align}\label{e:det Psi} 
\text{Det } D\Psi(x,y)=\phi(x,y)\frac{\partial \phi}{\partial y}(x,y)\frac{\partial \lambda}{\partial y}\left(x,\phi(x,y)\right)=\phi(x,y) \neq 0.
\end{align}
Let $\Omega_M^+=\Psi(\widetilde{W}_M^+)$ which is a domain in $\C^2$. Then $\Psi$ is a biholomorphism from $\widetilde{W}_M^+$ onto $\Omega_M^+$.

\begin{lem}\label{l:preliminary 2}
  Let $\widetilde{H'}$ be the lift of the map $H$ from $\widetilde{W}_M^+$ to $\Omega_M^+$, i.e., $\widetilde{H'}=\Psi \circ H \circ \Psi^{-1}$. Then there exists a holomorphic map $\tilde{Q}$ such that
  \[\widetilde{H'}(z,\z)=\left(\frac{a}{d}z+ \tilde{Q}(\z), \z^d\right).\]
  Also for $|\z| >MR_H$, $\displaystyle \tilde{Q}(\z)=\z^{d+d'}+A_{d+d'-1}\z^{d+d'-1}+\cdots+ A_0+\sum_{i=1}^\infty A_{-i}\z^{-i}$, where $d, d'$ are as defined before.
\end{lem}
\begin{proof}
Let $h_1(z,\z)=\pi_1 \circ \widetilde{H'}(z, \z)$ and $h_2(z,\z)=\pi_2 \circ \widetilde{H'}(z, \z)$ be the components of $\widetilde{H'}$. From the properties of the B\"{o}ttcher coordinate $\phi$ as discussed in subsection \ref{subsection 1}, it follows that for $(x,y) \in \widetilde{W}_M^+$
\[h_2 (\Psi(x,y))=\phi\circ H(x,y)=\phi(x,y)^d, \text{ i.e., } h_2(z,\z)=\z^d\]
for $(z,\z) \in \Omega_M^+$. In particular, $\frac{\partial h_2}{\partial \z}(z,\z)=d \z^{d-1}$ and $\frac{\partial h_2}{\partial z}(z,\z)=0.$ Thus
\[\text{Det } D\widetilde{H'}(z,\z)=d\z^{d-1}\frac{\partial h_1}{\partial z}(z,\z).\]
Now on $\widetilde{W}_M^+$, $\widetilde{H'}\circ \Psi =\Psi \circ H $. Hence  from (\ref{e:det Psi})
\[\text{Det } D\widetilde{H'}(\Psi(x,y))=\text{Det } D\widetilde{H'}(\psi(x,\phi(x,y)), \phi(x,y))=a \phi(x,y)^{d-1},\]
i.e., $\text{Det } D\widetilde{H'}(z,\z)=a \z^{d-1}=d\z^{d-1}\frac{\partial h_1}{\partial z}(z,\z).$ So for every $(z, \z) \in \Omega_M^+$
\[\frac{\partial h_1}{\partial z}(z,\z)=\frac{a}{d} \text{ or } h_1(z, \z)=\frac{a}{d}z+ \tilde{Q}(\z)\]
where $\tilde{Q}$ is holomorphic on $\Omega_M^+$. To prove that $\tilde{Q}$ has the desired form,
it suffices to show $Q(\z)/\z^{d+d'} \to 1$ as $|\z| \to \infty$. Since $E_2^{-1}(0, \z)=(0,\z)$, 
\begin{align*}
\tilde{Q}(\z)=h_1(0,\z)&=\psi \circ H \circ E_1^{-1}(0,\z)=\psi(H(0, \lambda(0,\z)))=\psi(P_1(\lambda(0,\z)),\phi \circ H(0,\lambda(0,\z)))\\
&=\psi\left(P_1(\lambda(0,\z)),(\phi(0,\lambda(0,\z))^d\right)=\psi\left(P_1(\lambda(0,\z)), \z^d\right)
\end{align*}
where $P_1(x)=\pi_1 \circ H(0,x)$, a monic polynomial of degree $d'=dd_m^{-1}$. Thus
\[\tilde{Q}(\z)=\psi\left(P_1\left(\lambda(0,\z)\right),\z^d\right)=\z^d \int_0^{P_1(\lambda(0,\z))} \frac{\partial \lambda}{\partial y}\left(t,\z^d\right)dt.\]
Now, from Remark \ref{r:preliminary 2} (ii) it follows that
$$ \frac{P_1(\lambda(0,\z))}{\z^{d'}}  \to 1$$ 
as $\z \to \infty$. Since $2d'\le d $, for $|\z|$ sufficiently large, $(t, \z^d) \in W_M^+$ for every $0\le |t|\le |P_1(\lambda(0,\z))|$. In particular, from Remark \ref{r:preliminary 2} (ii) again, it follows that as $\z \to \infty$, 
$$\frac{\partial\lambda}{\partial y}\left(t,\z^d\right) \to 1 \text{ whenever } |t|\le |P_1(\lambda(0,\z))|.$$
 Thus, as $\z \to \infty$,
\[\left|\frac{\tilde{Q}(\z)}{\z^{d+d'}}- \frac{P_1(\lambda(0,\z))}{\z^{d'}}\right| \to 0,\]
and consequently $\frac{\tilde{Q}(\z)}{\z^{d+d'}} \to 1$. This completes the proof.
\end{proof}

The Laurent expansion of $\tilde Q$ can be written as $\tilde{Q}(\z)=Q(\z)+Q^-(\z)$ where $Q$ is the principal part and $Q^-(\z)=\sum_{i=1}^\infty A_{-i}\z^{-i}$. Since $|\z|>MR_H $ for every $(z,\z) \in \Omega_M^+$ and $Q^-(\z)$ is defined there, the functions
\[
R_n(\z)=\sum_{i=0}^n \left(\frac{d}{a}\right)^{i+1}Q^-\left(\z^{d^i}\right)
\]
are well-defined for such $\z$. Then for every $|\z|>MR_H$, there exists $n_{\z} \ge 1$ and a constant $C>0$ such that 
\[|R_{n}(\z)-R_{n-1}(\z)|\le C\left|\frac{d}{a}\right|^{n}|\z|^{-d^n},\]
whenever $n \ge n_\z$. Hence $\{R_n\}$ converges locally uniformly to a holomorphic function, say $R$ on $\Omega_M^+$. Consider the injective map
$E_3$ on $\Omega_M^+$ as $E_3(x,y)=(x-R(y), y)$. Let $\widetilde{\Psi}=E_3 \circ \Psi$ and
$$\widetilde{\Omega}_M^+=E_3(\Omega_M^+)=\widetilde{\Psi} (\widetilde{W}_M^+).$$ 
Note that for $(z,\z) \in \widetilde{\Omega}_M^+$, 
\[
\left(\frac{a}{d}\right)R(\z)-R(\z^d)=Q^-(\z)
\]
and hence
\[\widetilde{\Psi} \circ H \circ \widetilde{\Psi}^{-1}(z,\z)=E_3 \circ \widetilde{H'} \circ E_3^{-1}(z,\z)=\left(\frac{a}{d}z+ Q(\z), \z^d\right),\]
which proves the following
\begin{prop}\label{p:main_subsect 2}
    Let $\widetilde{H}=\widetilde{\Psi} \circ H \circ \widetilde{\Psi}^{-1}$ on $\widetilde{\Omega}_M^+$. Then for $(z,\z) \in \widetilde{\Omega}_M^+$
    \[\widetilde{H}(z,\z)=\left(\frac{a}{d}z+ Q(\z), \z^d\right),\]
    where $Q(\z)=\z^{d+d'}+A_{d+d'-1}\z^{d+d'-1}+\cdots +A_0$.
\end{prop}

\subsection{\texorpdfstring{$\widehat{U}^+$}{} is biholomorphic to \texorpdfstring{$\cover$}{}}\label{subsection 3}
Recall from Proposition \ref{p:UedaBook} that $\widehat{\pi}$ restricted to $\widehat{V}^+$ is a biholomorphism onto its image. Let $\widehat{\pi}(\widehat{W}_M^+)=\widetilde{W}_M^+$ where 
$ \widehat{W}_M^+ \subset \widehat{V}^+ \subset \widehat{U}^+$ and $M>0$ is sufficiently large as fixed in the previous subsection. As $\widehat{H}$ is the lift of $H$ to $\widehat{U}^+$,
\begin{align} \label{e:iterative_relation 1}
H \circ \widehat{\pi}=\widehat{\pi} \circ \widehat{H}  \text{ and }H^n \circ \widehat{\pi}=\widehat{\pi} \circ \widehat{H}^n  .
\end{align}
for every $n \ge 1$. By Proposition \ref{p:main_subsect 2}, $H(\widetilde{W}_M^+) \subset \widetilde{W}_M^+$on $\w{W}_M^+$ and hence
\begin{align} \label{e:iterative_relation 2}
\widetilde{H}\circ \w{\Psi}= \w{\Psi} \circ {H} \text{ and }  \widetilde{H}^n\circ \w{\Psi} =\w{\Psi} \circ {H}^n
\end{align}
on $\widetilde{W}_M^+$ and for every $n \ge 1$. Thus, we observe the following:
\begin{prop}\label{p:preliminary 3}
Both $\left\{\widehat{H}^{-n}\left(\widehat{W}_M^+\right)\right\}$ and $\left\{\widetilde{H}^{-n}\left(\w\Omega_M^+\right)\right\}$ are increasing sequences of domains with
\[ \widehat{U}^+= \bigcup_{n=0}^\infty\widehat{H}^{-n}\left(\widehat{W}_M^+\right) \text{ and } \cover=\bigcup_{n=0}^\infty \widetilde{H}^{-n}\left(\w{\Omega}_M^+\right).\]
\end{prop}
\begin{proof}
Note that $I^-=[0:1:0]$ is a super-attracting fixed point of $H$ when extended to $\mathbb{P}^2$ with $U^+$ as its basin of attraction. Furthermore, $\w{W}_M^+$ corresponds to a neighbourhood of $I^-$ in $\mathbb{P}^2$. Thus, we have
\[U^+=\bigcup_{n=0}^\infty H^{-n}\left(\w{W}_M^+\right)\]
and $\left\{{H}^{-n}\left(\w{W}_M^+\right)\right\}$ is an increasing sequence. By (\ref{e:iterative_relation 1}), 
\[
H^n \circ \h{\pi} \circ \h{H}^{-n}\left(\h{W}_M^+\right)=\h{\pi}\left(\h{W}_M^+\right)=\w{W}_M^+\]
for every $n \ge 1$.
As $H$ is an automorphism of $\C^2$, $\h{\pi} \circ \h{H}^{-n}\left(\h{W}_M^+\right)= {H}^{-n}\left(\w{W}_M^+\right)$. Since $\left\{{H}^{-n}\left(\w{W}_M^+\right)\right\}$ is an increasing sequence, $\left\{\h{H}^{-n}\left(\h{W}_M^+\right)\right\}$ is also an increasing sequence of domains with $\displaystyle \widehat{U}^+= \bigcup_{n=0}^\infty\widehat{H}^{-n}\left(\widehat{W}_M^+\right)$. Consider $\C^2$ with coordinates $(z,\z)$ and the maps
\[P_1(z,\z)=(z,\z^d) \text{ and } P_2(z,\z)=\left(\frac{a}{d}z+Q(\z),\z\right).\]
 Note that $\w{H}=P_1 \circ P_2$ on $\w{W}_M^+$. Since $P_1$ is a proper self--map pf $\C^2$ and $P_2$ is an automorphism of $\C^2$, $\w{H}$ extends to a proper map on $\C^2$. Further, as $H\left(\w{W}_M^+\right) \subset \I{\w{W}_M^+}$, it follows from (\ref{e:iterative_relation 2}) that $\w{H}\left(\w{\Omega}_M^+\right)\subset \I{\w{\Omega}_M^+}$. Inductively, $\left\{\widetilde{H}^{-n}\left(\w\Omega_M^+\right)\right\}$ is an increasing sequence. Since $|\z|>MR_H$ in $\w{\Om}_M^+$, $\w{H}^{-n}\left(\w{\Om}_M^+\right) \subset \cover$ for every $n \ge 1$. Hence $\displaystyle\bigcup_{n=0}^\infty \widetilde{H}^{-n}\left(\w{\Omega}_M^+\right) \subseteq \cover$. To complete the reasoning, we first observe the following:

\begin{lem} \label{l:intermediate}
There exists $\w{M}>MR_H$ sufficiently large and $0<t<1$ such that 
$$S_{\w{M},t}=\left\{(z, \z) \in \C^2: |z|<t|\z|^2 \text{ and } |\z|\ge \w{M}\right\} \subset \w{\Om}_M^+.$$
\end{lem}
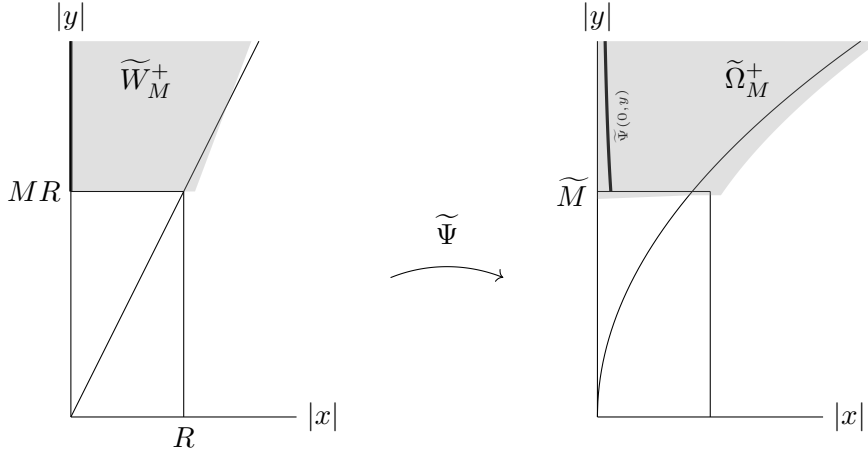
\begin{figure}[ht!]
\begin{tikzpicture}[scale=.5]
  \draw[-] (0,0) -- (6,0) node[right] {$|x|$};
  \draw[-] (0,0) -- (0,10) node[above] {$|y|$};

  \draw[-] (14,0) -- (20,0) node[right] {$|x|$};
  \draw[-] (14,0) -- (14,10) node[above] {$|y|$};

  \coordinate (O) at (0,0);
  \coordinate (R) at (3,0);
  \coordinate (MR) at (0,6);
  \coordinate(S) at (3,6);
  \coordinate (A) at (5,10);
  \coordinate (B) at (4.5,10);

   \coordinate (O1) at (14,0);
  \coordinate (R1) at (17,0);
  \coordinate (MR1) at (14,6);
  \coordinate(S1) at (17,6);
  \coordinate (A1) at (19,10);
  \coordinate (B1) at (18.5,10);
  
  \draw[thin] (O) -- (A);
  \draw[very thick, -] (0,6) -- (0,10);
  \draw[thin] (R) -- (S);
  \draw[thin] (MR) -- (S);

  \draw[thin, domain=0:10, smooth] plot ({.07*\x*\x + 14}, {\x});
  \draw[very thick, domain=6:10, smooth,-] plot ({5/\x^3+2/(\x) + 14}, {\x}); 
  \node[rotate=93, below, font=\tiny] at (14.15,8) {$\scriptstyle\w{\Psi}(0,y)$};
  \draw[thin] (R1) -- (S1);
  \draw[thin] (MR1) -- (S1);
\draw[ domain=8.5:11.5, smooth,->] plot ({\x}, {sqrt(16 - (\x - 10)^2)});
  \fill[black!40, opacity=0.3]  (0,6)--(5.5*.6, 6)-- (4.8,10)--(0,10)  -- cycle;
  \fill[black!40, opacity=0.3] (14,5.8) -- (.07*36+14,5.9) -- plot[domain=5.9:10, smooth]({.07*\x*\x + 14+5/\x}, {\x})-- (14,10)--cycle;
  \node[below] at (R) {$R$};
  \node[left] at (MR) {$MR$};
  \node[left] at (MR1) {$\w{M}$};
  \node at (2,9) {$\w{W}_M^+$};
  \node at (18,9) {$\w{\Om}_M^+$};
  \node at (10,5) {$\w{\Psi}$};
\end{tikzpicture}   
\caption{Behaviour of $\w{\Psi}$ on $\w{W}_M^+$ and the image of $\w{\Psi}(0,y)$.}
\label{f:cover}
\end{figure}

\begin{proof} 
Recall the function $R(\z)$ from the proof of Proposition \ref{p:main_subsect 2}. Choose $\w{M}>2MR_H$ sufficiently large, such that for $|\z|>\w{M}$, $|R(\z)|<\frac{1}{4M}|\z|^2$. Fix $t=\frac{1}{4M}$. Then for $(z,\z) \in S_{\w{M},t}$ 
\[\left|z+R(\z)\right|<\frac{1}{2M}|\z|^2 \text{ or } 2M \left|\frac{z+R(\z)}{\z}\right|<|\z|.\]
Hence $\left(\frac{z+R(\z)}{\z}, \z\right) \in W_{2M}^+ \subset W_M^+$. Recall the map $E_2$ from subsection \ref{subsection 2} that is defined on $W_M^+$. In particular, $E_2(x,y)=\p \circ E(x,y)$, where 
\[ E(x,y)=\left (\int_0^x \frac{\partial \lambda }{\partial y}(t,y)dt, y\right) \text{ and } \p(x,y)=(xy, y).\]
Note that $\p$ is defined on $\C^2$ and $E$ on $W_M^+$. Also, $\p$ is injective on $\C \times \C^*$ and for a fixed $|y|>MR_H$, the slice $D\left(0,\frac{|y|}{M}\right) \times \{y\} \subset W_M^+$. So consider the function
\[h_y(x)=\pi_1 \circ E(x,y)\]
that is well-defined on $D\left(0,\frac{|y|}{M}\right)$. By Remark \ref{r:preliminary 2}(ii) and by a similar argument as in (\ref{e:lambda}), 
\[|h_y(x)-x|\le \ep|x|<\frac{\ep}{M}|y| \text{ and }h_y(0)=0,\]
where $0<\ep<\frac{1}{2}$ as fixed there. Furthermore, $\ov{W_M^+} \subset V^+$ and $h_y$ extends continuously to $\partial D\left(0,\frac{|y|}{M}\right)$. Hence $h_y$ maps the boundary $\partial D(0, \frac{\vert y \vert}{M})$ into the annulus centered at the origin with inner radius $\frac{1-\ep}{M}|y|$ and outer radius $\frac{1+\ep}{M}|y|$. 

\medskip

We now claim that the entire disc $D\left(0, \frac{1-\ep}{M}|y|\right)$ is contained in $h_y \left( D\left(0,\frac{|y|}{M}\right)\right)$. To see this, note that there is a neighbourhood of the origin that is contained in the image $h_y \left( D\left(0,\frac{|y|}{M}\right)\right)$ as $h_y$ is a non-constant holomorphic map that fixes the origin. Let $\varrho >0$ be the largest disc centered at the origin that is contained in  $h_y \left( D\left(0,\frac{|y|}{M}\right)\right)$. Suppose that $\varrho < \frac{1-\ep}{M}|y|$. There must exist a $b_0$ with $\vert b_0 \vert = \varrho$ such that $b_0$ lies in the boundary of the image $h_y \left( D\left(0,\frac{|y|}{M}\right)\right)$. Choose $b_n \in h_y \left( D\left(0,\frac{|y|}{M}\right)\right)$ and pre-images $c_n \in h_y^{-1}(b_n) \cap D(0, \frac{\vert y \vert}{M})$. Any limit point $c_0$ of the $c_n$'s must lie on the boundary $\partial D(0, \frac{\vert y \vert}{M})$ by the open mapping theorem. This implies that $\vert h_y(c_0) \vert = \vert b_0 \vert = \varrho < \frac{1 - \ep}{M} \vert y \vert$ and this contradicts the fact that 
\[
\frac{1-\ep}{M}\vert y \vert \le \vert h_y(c_0) \vert \le \frac{1+ \ep}{M} \vert y \vert.
\]


Since $\ep<\frac{1}{2}$, it follows from the above that $D\left(0,\frac{|y|}{2M}\right) \times \{y\} \subset E(W_M^+)$ for every $|y|>MR_H$. Thus $E^{-1}\left(W_{2M}^+\right) \subset W_M^+$. In particular, if $(z, \z) \in S_{\w{M},t}$, where $\w{M}$ and $t$ are as fixed above, then 
\[ E^{-1}\left(\frac{z+R(\z)}{\z}, \z\right) \in W_M^+ \text{ or } (z,\z) \in E_3 \circ E_2 (W_M^+)=\w{\Om}_M^+.\qedhere\]
\end{proof}
Now, note that $\w{H}$ is defined on $\C^2$. So for every $(z, \z) \in \cover.$ 
\begin{align*}
 \w{H}^n(z,\z)&=\left(\left(\frac{a}{d}\right)^n z+\left(\frac{a}{d}\right)^{n-1}Q(\z)+\ldots+ \frac{a}{d}Q\left(\z^{d^{n-2}}\right) +Q\left(\z^{d^{n-1}}\right), \z^{d^n}\right),
\end{align*}
for every $n \ge 1$. Let  $h_n^z(\z)=\frac{\pi_1 \circ \w{H}^n(z,\z)}{\z^{2d^n}}$. 

\smallskip
{\it Claim:} For every $z \in \C$ and $|\z|>1$, $\displaystyle\lim_{n \to \infty} h_n^z(\z) \to 0 \text{ as } n \to \infty.$

\medskip 
Note that as $Q(\z)$ is a monic polynomial of degree $\z^{d+d'}$, where $d'=dd_m^{-1}$, there exists a constant $C_1>0$ such that
\[\left|\frac{Q(\z)}{\z^{2d}}\right|<\left|\frac{1}{\z^{d-d'}}\right|C_1\]
for every $|\z|>1$. Now fix a point $(z,\z) \in \cover$. Then for every $n \ge 1$
\begin{align}\label{e:hn}
    h_{n+1}^z(\z)=\frac{a}{d} \frac{h_n^z(\z)}{\z^{2d^n(d-1)}} + \frac{Q(\z^{d^n})}{\left(\z^{d^n}\right)^{2d}}.
\end{align}
In particular, for sufficiently large $n \ge 1$, $\left|\frac{a}{d} \z^{-2d^n(d-1)}\right|<1$, and
\begin{align*}
  \left|h_{n+1}^z(\z)\right|< \left|h_n^z(\z)\right| + \left|\frac{C_1}{\z^{d^n}}\right| \text{ or }  \Big|\left|h_{n+1}^z(\z)\right|-\left|h_n^z(\z)\right|\Big| \le \left|\frac{C_1}{\z^{d^n}}\right|.
\end{align*}
Thus $\{h_n^z(\z)\}$ is a bounded sequence, and from (\ref{e:hn}) there exist a constant $C>0$ such that
\[\left|h_n^z(\z)\right|< \left|\frac{C}{\z^{d^n}}\right|,\]
which proves the claim. Hence for sufficiently large $n$, $\w{H}^n(z,\z) \in S_{\w{M},t} \subset \w{\Om}_M^+$, where $\w{M}$ and $t$ are as chosen in the proof of Lemma \ref{l:intermediate}. This completes the proof of Proposition \ref{p:preliminary 3}.
\end{proof}
Finally, we conclude this subsection with the following result.
\begin{thm}
$\h{U}^+$ is biholomorphic to $\cover$. Consequently, the universal cover of $U^+$ is $\C \times \mathbb{D}$.
\end{thm}
\begin{proof}
Note $\w{\Psi}\circ \h{\pi}$ is a biholomorphism from $\h{W}_M^+$ to $\w{\Om}_M^+$ with 
\[\pi_2 \circ \w{\Psi}\circ \h{\pi}(p)=\hat{\phi}(p),\]
where $\hat{\phi}$ is as defined in subsection \ref{subsection 1} on $\h{U}^+$. Let $\hat{\psi}(p)=\pi_1 \circ \w{\Psi}\circ \h{\pi}(p).$ Then by Proposition \ref{p:main_subsect 2}, for $p \in \h{W}_M^+$
\[\left(\hat\psi\circ \h{H}(p),\hat\phi\circ \h{H}(p) \right)=\w{\Psi}\circ \h{\pi} \circ \h{H}(p)=\w{H}\circ \w{\Psi}\circ \h{\pi}(p)=\left(\frac{a}{d} \hat{\psi}(p)+Q\left(\hat{\phi}(p)\right), \hat{\phi}(p)^d\right).\]
For $p\in \h{H}^{-1}(\h{W}_M^+)$ define
\[\hat{\psi}(p)=\frac{d}{a} \left(\hat\psi\left(\h{H}(p)\right)-Q\left(\hat{\phi}(p)\right)\right).\]
Hence $\w{\Psi} \circ \h{\pi}$ extends to $\h{H}^{-1}(\h{W}_M^+)$ onto $\w{H}^{-1}(\w{\Om}_M^+)$. Now for $p,q \in \h{H}^{-1}(\h{W}_M^+)$ if
\[\w{\Psi} \circ \h{\pi}(p)=\w{\Psi} \circ \h{\pi}(q)\]
then $\hat\phi(p)=\hat{\phi}(q)$ and
\[\w{H}\circ \w{\Psi} \circ \h{\pi}(p)=\w{\Psi}\circ H \circ \h{\pi}(p)=\w{\Psi}\circ H \circ \h{\pi}(q)=\w{H}\circ \h{\Psi}\circ \w{\pi}(q).\]
As $\w{\Psi}$ is injective on $\w{W}_M^+$, it follows that
\[H \circ \h{\pi}(p)=H \circ \h{\pi}(q) \text{ or } \h{\pi}(p)=\h{\pi}(q).\]
Hence by Proposition \ref{p:UedaBook} we have $p=q$. Furthermore, from Proposition \ref{p:preliminary 3}, along with an induction argument it follows that $\w{\Psi} \circ \h{\pi}$ extends to a biholomorphism from $\h{U}^+$ to $\cover$. The second part is immediate as the universal cover of $\mathbb{D}\sm \{0\}$ is $\mathbb{D}$.
\end{proof}

\subsection{Deck transformations of \texorpdfstring{$\cover$}{}}\label{subsection 4} 
Note $\w{\Psi} \circ \h{\pi}$ is a biholomorphism between $\h{U}^+$ and $\cover$. Then $\h{\Pi}=\h{\pi} \circ \left(\w{\Psi} \circ \h{\pi}\right)^{-1} $ is the covering map. Since $\w{\Psi}$ and $\h{\pi}$ restricted to $\w{W}_M^+$ and $\h{W}_M^+$ respectively are biholomorphisms, it follows that $\h{\Pi}=\w{\Psi}^{-1}$ on $\w{\Om}_M^+$. Now by Proposition \ref{p:main_subsect 2}, $\w{H}$ is the lift of the map $H$, i.e., $H=\w{\Psi} \circ \w{H} \circ \w{\Psi}^{-1}$ on $\w{\Om}_M^+$ and the following diagram commutes.
\[\begin{tikzcd}
\cover \arrow{r}{\tilde{H}}   \arrow[swap]{d}{\widehat{\Pi}} & \cover  \arrow{d}{\widehat{\Pi}} \\
U^+ \arrow{r}{H} & U^+ 
\end{tikzcd}\]
The next steps are repetitions of the techniques used in \cite{Hu-Ov1} and \cite{BPV:IMRN}. However, we briefly revisit them, since $H$ is a generalised H\'{e}non map.
\begin{prop}\label{p:deck transformations}
For every $\left[\frac{k}{d^n}\right] \in \uppi_1(U^+)/\mathbb{Z}$, the deck transformation $\deck_{k/d^n}$ on $\cover$ satisfies the relation
\begin{align}\label{e:deck and H} 
\w{H} \circ \upgamma_{k/d^n}=\upgamma_{k/d^{n-1}} \circ \w{H}.
\end{align}
Thus, (\ref{e:deck transformations}) holds.
\end{prop}
\begin{proof}
As $\uppi_1(\w{U}^+)=\mathbb{Z}$, for every $\left[\frac{k}{d^n}\right] \in \uppi_1(U^+)/\mathbb{Z}$ there is a deck transformation $\h{\deck}_{k/d^n}$ such that $\h{\pi} \circ \h{\deck}_{k/d^n}=\h{\pi}$ on $\w{U}^+$. Furthermore, for $[p,C] \in \w{U}^+$ note that
\[ \h{\deck}_{k/d^n}\left([p,C]\right)=\left[p,C_{k/d^n}\right] \text{ such that } \alpha\left(CC_{k/d^n}^{-1}\right) \in \left[\frac{k}{d^n}\right],\]
where $\alpha$ is as defined in subsection \ref{subsection 1}. Let $\deck_{k/d^n}=\left(\w{\Psi} \circ \h{\pi}\right) \circ \h{\deck}_{k/d^n} \circ \left(\w{\Psi} \circ \h{\pi}\right)^{-1}.$ Fix a point $(z,\z) \in \cover$, and let $\left(\w{\Psi} \circ \h{\pi}\right)^{-1}(z,\z)=[p,C]$. Then 
\begin{align*}
    \w{H} \circ \deck_{k/d^n}(z,\z)&=\w{H} \circ \w{\Psi} \circ \h{\pi}\circ \h{\deck}_{k/d^n}\left([p,C]\right)=\w{H} \circ \w{\Psi} \circ \h{\pi}\left(\left[p,C_{k/d^n}\right]\right)\\
    &=\w{\Psi} \circ \h{\pi}\left(\left[H(p),lH\left(C_{k/d^n}\right)\right]\right).
\end{align*}
Now
\[ \alpha \left(H\left(CC_{k/d^n}^{-1}\right)\right)=\alpha \left(lH\left(C\right)H\left(C_{k/d^n}^{-1}\right)l^{-1}\right)\in \left[\frac{k}{d^{n-1}}\right].\]
Hence,
\begin{align*}
    \w{H} \circ \deck_{k/d^n}(z,\z)=\w{\Psi} \circ \h{\pi}\circ \h\deck_{k/d^{n-1}}\left(\left[H(p),lH\left(C\right)\right]\right)=\deck_{k/d^{n-1}}\circ \w{H}(z,\z).
\end{align*}
For $n=1$, by comparing the coefficients in the above equation, it follows that
\[\deck_{k/d}(z,\z)=\left(z+\frac{d}{a}\left(Q(\z)-Q(e^{2\pi ik/d}\z)\right), e^{2 \pi i k/d}\z\right).\]
Thus using (\ref{e:deck and H}) inductively (as in \cite[page 24-25]{BPV:IMRN}), the deck transformations are given by
\[
\deck_{k/d^n}( z,\z) = \left( z + \frac{d}{a} \sum_{l = 0}^{n-1} \left(\frac{d}{a}\right)^l \left( Q\left(\z^{d^l}\right) - Q \left( \left(e^{2\pi ik/d^n} \cdot \z\right)^{d^l}\right)\right), e^{2 \pi i k/d^n} \cdot \z \right).
\qedhere
\]
\end{proof}


\section{\texorpdfstring{${\rm Aut}(U^+)$}{Aut(U+)} and \texorpdfstring{${\rm Aut}(\Omega_c')$}{Aut(Om-c)}}

\noindent The goal of this section is to understand the structure of ${\rm Aut}_1(U^+)$. Note that maps in this subgroup admit lifts to $\cover$ and it is precisely this property that makes them amenable for further classification. Furthermore, we will use this observation, to compute  the automorphism group of the punctured {\it Short} $\C^2$'s, arising as the sub-level set of the Green's function of $H$, consisting of maps that induce identity on the fundamental group. Finally, patching these observations leads to a description of the automorphism groups of both $U^+$ and punctured {\it Short} $\C^2$'s.

\subsection{The subgroup \texorpdfstring{${\rm Aut}_1(U^+)$}{Aut1(U+)} } As mentioned before, we adapt the idea of Bousch \cite{Bousch} to prove the theorem given below. Similar ideas were also used to describe the automorphism group of punctured {\it Short} $\C^2$'s in \cite{BPV:IMRN} for simple H\'{e}non maps. Recall the polynomial $Q$, as constructed in the proof of Theorem \ref{t:HOv-gen} (also see Proposition \ref{p:main_subsect 2}). 
\begin{thm}\label{t:escaping set}
Let ${\rm Aut}_1(U^+)$ be as above. Then
 $$\mathbb{C}\le  {\rm Aut}_1(U^+)\le \mathbb{C}\rtimes \mathbb{Z}_{d_0(d-1)},$$
 where $d_0$ is an appropriate factor of $d+d'$.
\end{thm}
\begin{proof}  
The proof consists of several steps. The main ingredient is to understand the lifts of elements in ${\rm Aut}_1(U^+)$ to $\cover$ in a progressively better manner.

\medskip 
\noindent \textit{Step 1:} Let $\al \in {\rm Aut}_1(U^+)$. First, we claim that any lift of the map $\al$ to the covering space $\cover$ is of the form
\[A(z,\z)=(\beta z+\gamma (\z), \alpha \z),\]
where $\alpha , \beta \in \mathbb{C}$ with $|\alpha|,|\beta|=1$ and $\gamma$ is holomorphic.  Furthermore, for every pair of integers $k, n$, the lift $A$ commutes with the deck transformation $\deck_{k/d^n}$, i.e.,
\[ A \circ \deck_{k/d^n}(z,\z)=\deck_{k/d^n}\circ A(z,\z).\]

To see this, let $A(z, \z)=\left ({\rm A}_1(z, \z), {\rm A}_2(z, \z)\right).$ For a fixed $\z \in \mathbb{C} \sm \ov{\mathbb{D}}$, the map ${\rm A}_2(\cdot, \z)$ is an entire holomorphic function whose range is contained in $\mathbb{C} \sm \ov{\mathbb{D}}$. Hence, ${\rm A}_2(\cdot, \z)$ is a constant map and we can write ${\rm A}_2(z,\z)=h(\z)$. As $\al \in {\rm Aut}_1(U^+)$, both $\al$ and $\al^{-1}$ induce the identity on the fundamental group of $U^+$ and hence admit lifts to the covering space $\cover$. Let $A'$ be a lift of $\al^{-1}$. Then
\[\widehat{\Pi} \circ A \circ A'(z,\z)=\al \circ \widehat{\Pi} \circ A'(z,\z)=\al \circ \al^{-1} \circ \widehat{\Pi}(z,\z)=\widehat{\Pi}(z,\z),\]
where $\widehat{\Pi}$ is the covering map from $\cover$ to $U^+$ as in Subsection \ref{subsection 4}. Note that, both $A \circ A'$ and $A' \circ A$ are lifts of the identity map. Hence, there exist integers $k,k',n,n'$ such that
\[ A \circ A'=\deck_{k/d^n} \text{ and } A'\circ A=\deck_{k'/d^{n'}}.\]
Since $\deck_{k/d^n}$ and $\deck_{k'/d^{n'}}$ are deck transformations of the form (\ref{e:deck transformations}), both $A$ and $A'$ are automorphisms of $\cover$. Let 
\[A^{-1}=\left (\tilde{{\rm A}}_1(z, \z),\tilde{h}(\z)\right).\]
Thus, $\tilde{{\rm A}}_1 \circ (A_1(\cdot, \z),h(\z))=\textsf{Id}$ for every $\z \in \mathbb{C} \sm \ov{\mathbb{D}}$ and $h\circ \tilde{h}\equiv \tilde{h}\circ h\equiv \textsf{Id}$. Hence, ${\rm A}_1(\cdot, \z)$ is an automorphism of $\C$ and $h$ is an automorphism of $\C \sm \ov{\mathbb{D}}$. In particular, there exists an $\alpha$ with $|\alpha|=1$ such that
\begin{align}\label{e:A_1}
h(\z)=\alpha \z \text{ and }{\rm A}_1(z,\z)=\beta(\z) z+ \gamma(\z).
\end{align}
Now, choose $(z', \z')$ and $(z, \z)$ in the same fibre. Since $A$ is a lift of the automorphism $\al$, $A(z,\z)$ and $A(z', \z')$ also lie on the same fibre. Thus, by Theorem \ref{t:HOv-gen}, there exist deck transformations $\deck_{k/d^n}$ and $\deck_{k'/d^{n'}}$ such that
\[
\deck_{k/d^n}\circ A( z,\z)=A(z',\z')=A \circ \deck_{k'/d^{n'}}(z,\z).
\]
By comparing the second coordinate in the above equation, it follows that $\left[\frac{k}{d^n}\right]=\left[\frac{k'}{d^{n'}}\right]$ and hence
\[\deck_{k/d^n}\circ A( z,\z)=A \circ \deck_{k/d^{n}}(z,\z).\]
Also, from (\ref{e:deck transformations}),
\[ \deck_{k/d^n}(z,\z)=\left( z + \frac{d}{a} \sum_{l = 0}^{n-1} \left(\frac{d}{a}\right)^l \left( Q\left(\z^{d^l}\right) - Q \left( \left(e^{2\pi ik/d^n} \cdot \z\right)^{d^l}\right)\right), e^{2 \pi i k/d^n} \cdot \z \right)=(z', \z').\]
Hence $\z'=e^{2 \pi i k/d^n} \z$, and by equating the first coordinate, we get
\begin{align*}
z + \frac{d}{a} \sum_{l = 0}^{n-1} \left(\frac{d}{a}\right)^l \left( Q\left(\z^{d^l}\right) - Q \left( \left(e^{2\pi ik/d^n} \cdot \z\right)^{d^l}\right)\right)=z'.
\end{align*}
Again, by equating the first coordinate of $\deck_{k/d^n}\circ A( z,\z)=A(z',\z')$, we get
\begin{align*}
\beta(\z)z + \beta(\z)\frac{d}{a} \sum_{l = 0}^{n-1} \left(\frac{d}{a}\right)^l \left( Q\left(\alpha^{d^l}\z^{d^l}\right) - Q \left( \left(e^{2\pi ik/d^n} \cdot \alpha \z\right)^{d^l}\right)\right)+\gamma(\z)=\beta(\z')z'+\gamma(\z').
\end{align*}
Note that the above identities are true whenever both $(z,\z)$ and $(z',\z')$ lie in the same fibre. Thus, by comparing coefficients of $z$ on both sides, $\beta(\z)=\beta(\z')=\beta(e^{2 \pi i k/d^n}\z)$, for all integers $k, n \ge 1$. That is, $\beta$ is constant on the circle of radius $|\z|>1$. Hence, $\beta$ is a constant function. So, the above relation reduces to the following
\begin{align*}
&\beta \frac{d}{a} \sum_{l = 0}^{n-1} \left(\frac{d}{a}\right)^l \left( Q\left(\z^{d^l}\right) - Q \left( \left(e^{2\pi ik/d^n} \cdot  \z\right)^{d^l}\right)\right)+\gamma(\z')\\&=\frac{d}{a} \sum_{l = 0}^{n-1} \left(\frac{d}{a}\right)^l \left( Q\left(\alpha^{d^l}\z^{d^l}\right) - Q \left( \left(e^{2\pi ik/d^n} \cdot \alpha \z\right)^{d^l}\right)\right)+\gamma(\z).
\end{align*}
Also, $\z/\z'$ lies on the unit circle whenever $(z',
\z')$ is on the same fibre as $(z,\z)$. So, $\gamma(\z)-\gamma(\z')$ is uniformly bounded for a fixed $\z$. Thus, for every $k,n \ge 1$,

{\small \begin{align}\label{e:step 1}
\nonumber &\beta \frac{d}{a} \sum_{l = 0}^{n-1} \left(\frac{d}{a}\right)^l \left( Q\left(\z^{d^l}\right) - Q \left( \left(e^{2\pi ik/d^n} \cdot  \z\right)^{d^l}\right)\right)-\frac{d}{a} \sum_{l = 0}^{n-1} \left(\frac{d}{a}\right)^l \left( Q\left(\alpha^{d^l}\z^{d^l}\right) - Q \left( \left(e^{2\pi ik/d^n} \cdot \alpha \z\right)^{d^l}\right)\right)\\
=&\left(\frac{d}{a}\right)^n\left[\left(\beta-\alpha^{d^{n-1}(d+d')}\right)\left(1-e^{2\pi i kd'/d}\right)\z^{d^{n-1}(d+d')}\right]+R_{k,n}(\z)=\gamma(e^{2\pi i k/d^n}\z)-\gamma(\z),
\end{align}
}
where $R_{k,n}(\z)$ is a polynomial of degree at most $d^{n}$. Let $ a_n= \alpha^{d^{n-1}(d+d')}.$

\medskip
{\it Claim: } $a_n=\beta$ for $n$ sufficiently large.

\medskip The proof of the claim is the same as discussed in \cite[Page 27]{BPV:IMRN}. We briefly revisit the steps. Fix a $k=1$. Let
\begin{align}\label{e:Q and L}
L_n(\z)=Q(\z)-Q(e^{2\pi i/d^n}\z) \text{ and } \w{L}_n^l(\z)=\beta L_{n-l}(\z)-L_{n-l}(\alpha^{d^l}\z)
\end{align}
where $n \ge 1$ and $0 \le l \le n-1$. Then (\ref{e:step 1}) can be expressed as
\begin{align}\label{e:L and gamma}
\sum_{l=0}^{n-1}\left(\frac{d}{a}\right)^{l+1} \w{L}_n^l\left(\z^{d^l}\right)= \gamma \left(e^{2\pi i/d^n}\z\right)-\gamma(\z).
\end{align}
Since $Q$ is a monic polynomial of degree $d+d'$ and $|\alpha|=1$, both $L_n$ and $\w{L}_n^l$ are polynomials of degree at most $d+d'$ and their coefficients are uniformly bounded for every $n\ge 1$ and $0 \le l \le n-1$. Let 
\[\w{L}_n^l(\z)=M_n^l \z^{d+d'}+\text{ l.o.t.} \] Then choose $M>1$ such that 
 \[\left|\w{L}_n^l(\z)\right|<M|\z|^{d+d'} \text{ and } \left|\w{L}_n^l(\z)-M_n^l \z^{d+d'}\right|<M|\z|^{d+d'-1}\]
 for $|\z|> 1$. 
Note that $M_n^{n-1}=(\beta-a_n)(1-e^{2\pi id'/d})$ and 
\[
a_n=\left(\alpha^{d+d'}\right)^{d^{n-1}}=f^{n-1}\left(\alpha^{d+d'}\right)
\]
where $f(z)=z^d$. Thus, either $\{a_n\}$ is dense in the unit circle or $a_n=a_n^{d^l}=a_{n+l}$ for some $n\ge n_0$ and $l \ge 1$, i.e., $a_0$ is a pre-periodic fixed point of $f$. Suppose that the sequence $\{a_n\}$ does not converge to $\beta$. Then in either of the above cases, there exists a convergent subsequence say $\{a_{n_l}\}$ such that for every $l \ge 1$,
\[|a_{n_l}-\beta|>\w{\delta}, \text{ i.e., }\lim_{l \to \infty} M_{n_l}^{n_l-1}=\delta \neq 0.\]
Now from (\ref{e:Q and L}) and (\ref{e:L and gamma}), 
\begin{align*}
    \left|\gamma \left(e^{2\pi i/d^n}\z\right)-\gamma(\z)-\left(\frac{d}{a}\right)^{n} M_n^{n-1}\z^{d^{n-1}(d+d')}\right|\le \max\left\{1,\left|\frac{d}{a}\right|^n\right\} nM|\z|^{d^{n-1}(d+d'-1)}.
\end{align*}
for $|\z|>1$. By the ratio test, for a fixed $|\z|>1$, both $\z^{d^{n-1}} \to \infty$ and $\left(\frac{d}{a}\right)^n \z^{d^{n-1}} \to \infty$ as $n \to \infty.$ Also $\gamma \left(e^{2\pi i/d^n}\z\right)-\gamma(\z)$ is bounded for every $n \ge 1$. Hence, we have
\begin{align*}
0 \neq |\delta|=\lim_{l \to \infty} \left|M_{n_l}^{n_l-1}\right|= &\lim_{l\to \infty}\left|\frac{\gamma \left(e^{2\pi i/d^{n_l}}\z\right)-\gamma(\z)} {\left(\frac{d}{a}\right)^{n_l} \z^{d^{n_l-1}(d+d')}}-M_{n_l}^{n_l-1}\right| \\ \le& \lim_{l \to \infty} \max\left\{1,\left|\frac{d}{a}\right|^{n_l}\right\} \left|\frac{a}{d}\right|^{n_l}\frac{n_lM}{\left| \z^{d^{n_l-1}}\right|}=0,
\end{align*}
which is a contradiction! Thus $a_n \to \beta$, i.e., $a_0$ is pre-periodic. In particular, $\beta=\alpha^{d^n(d+d')}=\alpha^{d^{n+1}(d+d')}=\beta=\beta^{d}$ for $n \ge n_0$. Hence $\beta^{d-1}=1\text{ and }\alpha^{d^{n_0}(d+d')(d-1)}=1.$

\medskip Let $d=\pr_1^{m_1}\pr_2^{m_2}\hdots \pr_l^{m_l}$ for primes $\pr_i, 1 \le i \le l$ and $m_i \ge 1$. Note that
\[\mathbb{Z}\left[\frac{1}{d}\right]=\left\{k\pr_1^{n_1}\pr_2^{n_2}\hdots \pr_l^{n_l}:k \in \mathbb{Z}, n_i \in \mathbb{Z}, 1 \le i \le l \right\}.\]
Define 
\begin{align}\label{d0}
d_0=\inf\left\{ (d+d')r \in \mathbb{N}: r =\pr_1^{-n_1}\pr_2^{-n_2}\hdots \pr_l^{-n_l}, n_i \in \mathbb{N}\cup \{0\}, 1 \le i \le l\right\}.
\end{align}
Note that $(d+d')d_0^{-1} \in \mathbb{Z}\left[\frac{1}
{d}\right]$ and is a positive integer. In particular, $d_0$ is a factor of $(d+d')$. Observe that if $H$ is a simple H\'{e}non map then $d'=1$ and $d_0=d+1$.

\medskip\noindent\textit{Step 2:} For every $\al \in {\rm Aut}_1(U^+)$ there exists a {\it unique} lift to the covering space $\cover$ of the form
\[\w{A}(z, \z)=(\tilde{\beta} z+\tilde{\gamma}(\z), \tilde{\alpha} \z),\]
where $\tilde{\alpha}^{d^{n-1}(d+d')}=\tilde{\beta}$ for every $n \ge 1$ and $\tilde{\beta}^{d-1}=\tilde{\alpha}^{d_0\big(d-1\big)}=1$.

\medskip To see this, let $A$ be an arbitrary lift of the automorphism $\al$. By {\it Step 1}, 
\[ A(z, \z)=(\beta z+\gamma (\z), \alpha \z).\] Then by the above step there exists $k_0\ge 1$ and $n_0 \ge 1$ such that
\[ \alpha =e^{{2\pi i k_0}/{d^{n_0-1}}(d+d')(d-1)}.\]
Since none of the $\pr_i, 1 \le i \le l$ can be the factors of both $d_0$ and $d-1$, it follows that $d_0(d-1)$ is coprime to $(d+d')d_0^{-1}$ and $d^n$ for every $n \ge 1$. Thus, there exist integers $p,q$ such that
\[pd^{n_0-1}(d+d')d_0^{-1}+q(d-1)d_0=k_0. \]
In particular,
\[\frac{k_0}{d^{n_0-1}(d+d')(d-1)}=\frac{q}{d^{n_0-1}(d+d')d_0^{-1}}+\frac{p}{d_0(d-1)}.\]
Now by construction of $d_0$, 
\[
\frac{q'}{d^{n_0'}}=\frac{q}{d^{n_0-1}(d+d')d_0^{-1}} \in \mathbb{Z}\left[ \frac{1}{d}\right], 
\]
i.e., there exists a deck transformation $\deck_{q'/d^{n_0'}}$. Thus, $\w{A}=A \circ \deck_{q'/d^{n_0'}}^{-1}$ is a lift of the automorphism $\al$, which satisfies the properties stated in {\it Step 2}.

\medskip
To prove the uniqueness of the lift, let $\h{A}$ and $\w{A}$ be lifts of $\al$ such that
\[ \h{A}(z,\z)=(\hat{\beta} z+\hat{\gamma}(\z),\hat{\alpha} \z ) \text{ and } \w{A}(z,\z)=(\tilde{\beta}z+\tilde{\gamma}(\z),\tilde{\alpha} \z )\]
with $\hat{\alpha}^{d_0(d-1)}={\tilde{\alpha}}^{d_0(d-1)}=1=\hat{\beta}^{d-1}=\tilde{\beta}^{d-1}$, and  $\hat{\alpha}^{(d+d')}=\hat{\beta}$, $\tilde{\alpha}^{(d+d')}=\tilde{\beta}$. Since $\h{A}\circ \w{A}^{-1}$ is a lift of the identity map, it corresponds to a deck transformation. Assume,
\[ \alpha=e^{2 \pi i k_1/d_0(d-1)} \text{ and } \tilde{\alpha}=e^{2 \pi i k_2/d_0(d-1)}\z\]
for some integers $k_1, k_2$.Thus
\[\h{A} \circ \w{A}^{-1}(z,\z)= \left(\hat{\beta} \tilde{\beta}^{-1}z+ h(\z), e^{2 \pi i (k_1-k_2)/d_0(d-1)}\z\right), \]
i.e., $\beta=\tilde{\beta}$ and $\frac{k_1-k_2}{d_0(d-1)} \in \mathbb{Z}\left[ \frac{1}{d}\right]$. Since $d_0(d-1)$ is coprime to $d$, it follows that $k_1-k_2 \in \mathbb{Z}$ and consequently, $\h{A}\circ \w{A}^{-1}=\textsf{Identity}$. This completes {\it Step 2}.

\medskip\noindent \textit{Step 3:} Let $\al \in {\rm Aut}_1(U^+)$ be such that it admits a lift of the form
\[ A(z,\z)=(z+\gamma(\z), \z).\]
Then $\gamma\equiv c$ for some $c \in \mathbb{C}$, i.e., $\gamma$ is a constant function.

\medskip To see this, note that since $A \circ \deck_{k/d^n}=\deck_{k/d^n}\circ A$ for every integer $k,n$, it follows that for a fixed $\z \in \cover$, 
\[ \gamma(e^{2\pi ik/d^n}\z)=\gamma(\z).\]
Thus, $\gamma$ is constant on the circle of radius $|\z|$, and is therefore identically constant on $\mathbb{C}\sm \ov{\mathbb{D}}$.

\medskip\noindent{\it Step 4:} Let $A=(z+c, \z)$, $c \in \C$. Then there exists $\al \in \textrm{Aut}_1(U^+)$ such that it admits $A$ as a lift.

\medskip Define the map $\al$ on $U^+$ as 
\[\al \circ \widehat{\Pi}(z,\z)=\widehat{\Pi}\circ A(z, \z).\]
where $\widehat{\Pi}$ is the covering map from $\cover$ to $U^+.$ Note that $\al$ is well defined if the image of points in the same fibre is the same. To check this, let $(z, \z)$ and $(z',\z')$ be points in the same fibre. Then there exist $k,n \in \mathbb{Z}$ such that $\deck_{k/d^n}(z, \z)=(z',\z')$. As $A \circ \deck_{k/d^n}=\deck_{k/d^n} \circ A$
\begin{align*}
	\al \circ \widehat{\Pi}(z',\z')=\widehat{\Pi}\circ A \circ \deck_{k/d^n}(z, \z)
=\widehat{\Pi}\circ \deck_{k/d^n}\circ A(z, \z)=\al \circ \widehat{\Pi}(z, \z).
\end{align*}
Thus, by {\it Step 3} and {\it Step 4}, $\C \le {\rm Aut}_1(U^+)$.

\medskip\noindent \textit{Step 5:} $\mathbb{C} \le  \textrm{Aut}_1(U^+) \le \C \rtimes \mathbb{Z}_{d_0(d-1)}$. In particular, there exists a subgroup $G$ of $\mathbb{Z}_{d_0(d-1)}$ and an operation
\begin{align*} 
\bullet : (\C \rtimes G) &\times (\C \rtimes G)\to \C \rtimes G\\
(\gamma_1, \alpha_1 ) &\bullet (\gamma_2, \alpha_2 )=\left(\gamma_1+\phi_{\alpha_1}(\gamma_2),  \alpha_1\alpha_2\right),
\end{align*}
where $\phi_{\alpha}: \C \to \C$ is a group isomorphism for every $\alpha \in G$ such that 
\[\left({\rm Aut}_1(U^+),\circ\right) \cong (\C \rtimes  G, \bullet).\]

\medskip To prove this, let $\al \in {\rm Aut}_1(U^+)$. By {\it Step 2} there exists a unique lift 
\[ A(z, \z)=(\beta z+\gamma (\z), \alpha \z),\]
such that $\alpha^{d_0(d-1)}=\beta^{d-1}=1$. Fix a base point $\z_0$ in $\C \sm \ov{\mathbb{D}}$ and define the map 
\[ \G: {\rm Aut}_1(U^+) \to \C \times \mathbb{Z}_{d_0(d-1)}\]
by $\G(\al)=(\gamma(\z_0), \alpha)$ where $A$ is the lift as assumed above. Note that $\G$ is well-defined, since by {\it Step 2} the map $A$ is a unique lift of the automorphism $\al$.

\medskip Suppose $\G(\al)=\G(\al')=(c, \alpha).$ Then there exist lifts of $\al$ and $\al'$, say $A$ and $A'$  such that
\[A(z,\z)=(\alpha^{d+d'}z+\gamma(\z), \alpha \z) \text{ and } A'(z,\z)=(\alpha^{d+d'}z+\gamma'(\z), \alpha \z)\]
with $\gamma(\z_0)=\gamma'(\z_0)=c$. Also, there exists a lift of the map $\al' \circ \al^{-1}$ of the form
$$A'\circ A^{-1}(z,\z)=(z+\gamma'(\alpha^{-1}\z)-\gamma(\alpha^{-1}\z),\z).$$
By {\it Step 3}, $\gamma'(\alpha^{-1}\z)-\gamma(\alpha^{-1}\z)\equiv c$ is a constant function. However, by assumption $\gamma(\z_0)=\gamma'(\z_0)$, i.e., $A'\circ A^{-1}(z,\alpha \z_0)=(z,\alpha \z_0)$. Hence $c=0$, or equivalently $\al=\al'$.

\medskip Let 
$$G=\{\alpha: \text{ there  exists } \al \in {\rm Aut}_1(U^+) \text{ such that } \G(\al)=(c,\alpha) \text{ for some } c \in \mathbb{C}\}.$$
Furthermore, note that if $\alpha, \tilde{\alpha} \in G$ then both $\alpha^{-1}$ and $\alpha \tilde{\alpha}$ are in $G$, and hence $G$ is a subgroup of  $\mathbb{Z}_{d_0(d-1)}$. Also, by {\it Step 4}, $\mathbb{C} \times G =\G\left({\rm Aut}_1 (U^+)\right)$. Thus, $\G$ induces the semi-direct product $\mathbb{C} \rtimes G$, and this completes the proof.
\end{proof}

\subsection{The group \texorpdfstring{${\rm Aut}_1(\Omega_c')$}{}}\label{ss:Short C2} Recall that 
\[ \Omega_c=\{z \in \C^2: G_H^+(z)<c\}, \;\;c>0\]
is a {\it Short} $\C^2$.
A punctured {\it Short} $\C^2$, denoted by $\Omega_c'$, is the complement of the escaping set in $\Omega_c$, i.e., $\Omega_c'=\Om_c \setminus K^+$.
Next, we give a generalised version of \cite[Theorem 1.2]{BPV:IMRN} using the fact that the covering maps over punctured {\it Short} $\C^2$'s are restrictions of $\widehat{\Pi}$ on appropriate invariant subspaces of $\cover$. As before, we will use ${\rm Aut}_1(\Omega_c')$ to denote the automorphisms of $\Omega_c'$ that induce the identity on the fundamental group of $\Omega_c'$. The factor $d_0$ as in (\ref{d0}) will continue to appear in the sequel. 
\begin{thm}\label{t:revised BPV}
Let $\mathcal{A}_c=\{z \in \C: 1 < |z|<e^c\}$. The map $\widehat{\Pi}$ restricted to $\C \times \mathcal{A}_c$ for every $c>0$ is a covering over $\Omega_c'$ and the following diagram commutes
\[\begin{tikzcd}
\mathbb{C}\times \mathcal{A}_c \arrow{r}{\tilde{H}}   \arrow[swap]{d}{\widehat{\Pi}} & \mathbb{C}\times \mathcal{A}_{dc} \arrow{r}{\tilde{H}} \arrow{d}{\widehat{\Pi}} &  \mathbb{C}\times \mathcal{A}_{d^2c} \arrow{r} {\tilde{H}}  \arrow{d}{\widehat{\Pi}}& \cdots\\
\Omega_c' \arrow{r}{H} & \Omega_{dc}' \arrow{r} {H} & \Omega_{d^2c}'  \arrow{r}{H} & \cdots 
\end{tikzcd}\]
Furthermore, 
\begin{enumerate}
	\item [(i)] The fundamental group $\uppi_1(\Omega_c')$ is isomorphic to $\mathbb{Z}\left[\frac{1}{d}\right]$ for every $c>0$.
	\item [(ii)] An element $\left[\frac{k}{d^n}\right]$ of $\uppi_1(\Omega'_c)/\mathbb{Z}$ corresponds to the following deck transformation of $\widehat{\Omega'_c}$: 
\[
\deck_{k/d^n}( z,\z) = \left( z + \frac{d}{a} \sum_{l = 0}^{\infty} \left(\frac{d}{a}\right)^l \left( Q\left(\z^{d^l}\right) - Q \left( \left(e^{2\pi ik/d^n} \cdot \z\right)^{d^l}\right)\right), e^{2 \pi i k/d^n} \cdot \z \right).
\]	

\smallskip
\item[(iii)] Also, $\mathbb{C}\le {\rm Aut}_1(\Omega_c')\le \mathbb{C}\rtimes \mathbb{Z}_{d_0(d-1)}$.
\end{enumerate}
\end{thm}
\begin{proof}
Recall that $\uppi_1(U^+)=\mathbb{Z}\left[\frac{1}{d}\right]$. Let 
$$\mathcal{D}=\left\{\deck_{{k}/{d^n}}: \left[\frac{k}{d^n}\right] \in \mathbb{Z}\left[\frac{1}{d}\right]/\mathbb{Z}\right\}$$ 
be the group of deck transformations.
Note that $\deck_{k/d^m}(\C \times \mathcal{A}_c)=\C \times \mathcal{A}_c.$ Also, the action of $\mathcal{D}$ is properly discontinuous on $\C \times \mathcal{A}_c.$ Hence 
\[(\C \times \mathcal{A}_c)/\mathcal{D} \simeq \widehat{\Pi}(\C \times \mathcal{A}_c).\]
Let $\tilde{\Omega}_c=\widehat{\Pi}(\C \times \mathcal{A}_c)$. Thus $\C \times \mathcal{A}_c$ is a covering of $\tilde{\Omega}_c$ via the map  $\widehat{\Pi}$ and $\uppi_1(\tilde{\Omega}_c)=\mathbb{Z}\left[\frac{1}{d}\right]$. Next, we use Proposition \ref{p:UedaBook} to prove that:

\medskip\noindent {\it Claim: }$\tilde{\Omega}_c=\Omega_c'$ for every $c>0$.

\medskip
Let $z \in \tilde{\Omega}_c$ and choose $p=[z,C] \in \widehat{U}^+$. Since $z \in U^+$, $H^n(z) \in V_R^+$ and $H^n(C) \subset V_R^+$, i.e., $\widehat{H}^n(p) \in  \widehat{V}^+$ and $\tilde{H}^n(\C \times \mathcal{A}_c)=\C \times \mathcal{A}_{cd^n}$ for every large enough $n\ge 1$. Hence, $\widehat{\Phi}(\widehat{H}^n(p)) \in \C \times \mathcal{A}_{cd^n}.$ Thus 
\[ \left|\hat{\phi}(\widehat{H}^n(p))\right|<e^{cd^n} \text{ or } \left|\hat{\phi}(p)^{d^n}\right|<e^{cd^n}.\]
In particular, $|\hat{\phi}(p)|<e^c$. Also, $\hat{\phi}(\widehat{H}^n(p))=\phi(H^n(z))$ by Proposition \ref{p:UedaBook} (i) and
\[G_H^+(H^n(z))=\log|\phi(H^n(z))|<d^n c\]
for every $n \ge 1$ large enough. Thus $G_H^+(z)< c$ for every $z \in \tilde{\Omega}_c$, and consequently we have $\tilde{\Omega}_c \subset \Omega_c'$. 

\medskip
Conversely, let $z \in \Omega_c'$ and choose $p=[z,C] \in \widehat{U}^+$. Note that for sufficiently large $n$, 
\[G_H^+(H^n(z))=\log|\phi(H^n(z))|<d^n c,\]
or equivalently $G_H^+(z)<c$. Hence again by \ref{p:UedaBook} (i), $\widehat{\Phi}(\widehat{H}^n(p)) \in \C \times \mathcal{A}_{d^n c}$, or equivalently $\tilde{H}^n \circ \widehat{\Phi}(p)\in \C \times \mathcal{A}_{d^n c}$. Thus, $\widehat{\Phi}(p) \in \C \times \mathcal{A}_{c}$ and $z \in \widehat{\Pi}(\C \times \mathcal{A}_c)=\tilde{\Omega}_c$.

\medskip Further, $\uppi_1(\Omega_c')=\mathbb{Z}\left[\frac{1}{d}\right]$, and for every $[\frac{k}{d^n}] \in \pi_1(\Omega_c')/\mathbb{Z}$ 
\[\widehat{\Pi} \circ \deck_{k/d^n}=\widehat{\Pi},\]
where $\deck_{k/d^n}$ is as in Theorem \ref{t:HOv-gen}. Hence, the deck transformations of $\mathbb{C} \times \mathcal{A}_c$ are given by
\[
\deck_{k/d^n}( z,\z) = \left( z + \frac{d}{a} \sum_{l = 0}^{\infty} \left(\frac{d}{a}\right)^l \left( Q\left(\z^{d^l}\right) - Q \left( \left(e^{2\pi ik/d^n} \cdot \z\right)^{d^l}\right)\right), e^{2 \pi i k/d^n} \cdot \z \right).
\]	 
Now, by the reasoning used in Theorem 1.2 in \cite{BPV:IMRN} and Corollary \ref{c:punctured short} above, the statement on ${\rm Aut}_1(\Omega_c')$ holds. It is indeed also similar to the proof of Theorem \ref{t:escaping set}; however, we will summarize the steps for the sake of completeness.

 \medskip Let $c>0$. By the reasoning in {\it Step 1} in the proof of Theorem \ref{t:escaping set}, the lift of any automorphism $\al \in {\rm Aut}_1(\Omega_c')$ to $\C \times \mathcal{A}_c$ has the form
\[A(z, \z)=\left( \tilde{\beta}(\z) z+\tilde{\gamma}(\z), h(\z)\right)\]
where $\tilde{\gamma}, \tilde{\beta} $ are holomorphic functions on $\mathcal{A}_c$ and $h$ is an automorphism of $\mathcal{A}_c$. Further, $\tilde{\beta}(\z) \neq 0$ for every $\z \in \mathcal{A}_c$.

\medskip {\bf Case 1: }Suppose $h(\z)=\tilde{\alpha}e^c/\z $  for some $|\tilde{\alpha}|=1$. Now, choose $(z', \z')$ and $(z, \z)$ on the same fibre. Then, by assumption $A(z,\z)$ and $A(z', \z')$ lies on the same fibre. In particular, for every deck transformation $\deck_{k/d^n}$ there exists a corresponding transformation $\deck_{k'/d^{n'}}$ such that $\deck_{k/d^n}\circ A( z,\z)=A(z',\z')=A \circ \deck_{k'/d^{n'}}(z,\z)$. By comparing the second coordinate, it follows that $\left[\frac{k}{d^n}\right]=-\left[\frac{k'}{d^{n'}}\right]$ or $k+k'=md^n$ for some integer $m \ge 1$. Hence,
\[\deck_{k/d^n}\circ A( z,\z)=A \circ \deck_{(md^n-k)/d^{n}}(z,\z) \text{ or } \deck_{k/d^n}\circ A\circ\deck_{k/d^n} ( z,\z)=A(z,\z) .\]
Again, by comparing the first coordinate of the above expression first we have 
\begin{align*}
  \tilde{\beta}\left(e^{2\pi ik'/d^n}\z\right)z=\tilde{\beta}(\z)z.
\end{align*}
Since the above is true for every pair of positive integers $k', n \ge 1$, $\tilde{\beta}$ is a non-zero constant map. Further,
\begin{align}\label{e:automorphism of annulus}
  \nonumber &\tilde{\beta} \sum_{l = 0}^{n-1} \left(\frac{d}{a}\right)^{l+1} \left( Q\left(\z^{d^l}\right) - Q \left( \left(e^{2\pi ik'/d^n} \cdot \z\right)^{d^l}\right)\right)+\tilde{\gamma}\left(e^{2\pi ik'/d^n}\z\right)\\
   =& \sum_{l = 0}^{n-1} \left(\frac{d}{a}\right)^{l+1} \left( Q\left({\tilde{\alpha}e^c/\z}^{d^l}\right) - Q \left( \left(e^{2\pi ik'/d^n} \cdot {\tilde{\alpha}e^c/\z}\right)^{d^l}\right)\right)+\gamma(\z).
\end{align}
Now, by a similar computation as in {\it Step 1} of the proof of Theorem \ref{t:escaping set}, for a fixed $\z \in \C \sm \ov{\mathbb{D}}$, the left side above in (\ref{e:automorphism of annulus}) diverges to infinity. However, the right side in (\ref{e:automorphism of annulus}) is bounded, which is a contradiction. Hence, $h(\z)\neq \tilde{\alpha}e^c/\z $  for some $|\tilde{\alpha}|=1$.

\medskip {\bf Case 2: }Suppose $h(\z)=\tilde{\alpha}\z $ for some $|\tilde{\alpha}|=1$, then by the exactly same reasoning as in {\it Step 1} and {\it Step 2} of the proof of Theorem \ref{t:escaping set}, it follows that there exists a unique lift of every $\al \in {\rm Aut}_1(\Omega_c')$ to $\C \times \mathcal{A}_c$ of the form
\begin{align}\label{e:unique lift}
A(z, \z)=\left( \beta z+\gamma(\z), \alpha \z \right),
\end{align}
where $\alpha^{d+d'}=\beta$, $\beta^{d-1}=\alpha^{d_0(d-1)}=1$ and $\gamma$ is a holomorphic function on $\mathcal{A}_c$. Here $d_0$ is as defined by (\ref{d0}) in the proof of Theorem \ref{t:escaping set}.

\medskip In particular, as in {\it Step 5} of Theorem \ref{t:escaping set}, for every $c>0$ there exists an injective map from 
\[ \G_c: {\rm Aut}_1(\Omega_c') \to \C \times \mathbb{Z}_{d_0(d-1)},\]
given by $\G_c(\al)=(\gamma(\z_c), \alpha)$ where $\z_c$ is a fixed point in $\mathcal{A}_c$. Now by repeating {\it Steps 3,4 and 5} from the proof of Theorem \ref{t:escaping set}, the result follows.
\end{proof}

\begin{rem}\label{r:short c2 and Gc}
Let $f \in {\rm Aut}_1(\Omega_c')$ such that the $\G_c(f)=(a_c,\alpha)$. Then for every $0<b<c$ and $(z,\z) \in \C \times \mathcal{A}_b$, 
\[f \circ \h{\Pi}(z,\z)=\h{\Pi}(\alpha^{d+d'}z+\gamma(\z), \alpha \z) \text{ where }\gamma(\z_c)=a_c. \]
Hence $f \in {\rm Aut}_1(\Omega_b')$ and $\G_b(f)=(a_b, \alpha)$ where $\gamma(\z_b)=a_b$. In particular, if $\alpha=1$ then $\gamma$ is a constant function, hence $a_c=a_b$. Thus for every $a \in \C$ and $0<b<c$
\[ \G_c^{-1}(a,1)=\G_b^{-1}(a,1)=(z+a,\z).\]
Similarly, $\G^{-1}(a,1)=\G_c^{-1}(a,1)$ for every $c>0$ and $a \in \C$, where $\G$ is as defined in {\it Step 5} of Theorem \ref{t:escaping set}.
\end{rem}

\subsection{The automorphism groups of \texorpdfstring{$U^+$}{U+} and \texorpdfstring{$\Om_c'$}{Om-c}}
 As a consequence of Theorem \ref{t:escaping set} and \ref{t:revised BPV}, we prove the following results on the entire automorphism group of $U^
 +$ and $\Omega_c'$ for every $c>0.$
\begin{thm}\label{t:aut U+}
Let $H$ be a map of the form (\ref{e:henon}). Then  
\begin{itemize}
    \item[(i)] $\C \times \mathbb{Z}\le {\rm Aut}(U^+) \le \mathbb{C}\rtimes \mathbb{Z}_{d_0(d-1)} \times \mathbb{Z}$ if $d=\pr^m$ for a prime number $\pr$ and a positive integer $m \ge 1$.
    \item[(ii)]$\mathbb{C} \times \mathbb{Z}\le {\rm Aut}(U^+)\le \left(\mathbb{C}\rtimes \mathbb{Z}_{d_0(d-1)}\right)\times \mathbb{Z}^{l}$
if $d=\pr_1^{m_1}\hdots\pr_l^{m_l}$ for primes $\pr_i, 1 \le i \le l$, and $m_i \ge 1$.
\end{itemize}
\end{thm}
\begin{proof}
 Let $d=\pr_1^{m_1}\hdots\pr_l^{m_l}$ for primes $\pr_i, 1 \le i \le l$, and $m_i \ge 1$. Recall that, the fundamental group of $U^+$ is 
\[\uppi_1(U^+)\simeq\mathbb{Z}\left[\frac{1}{d}\right]=\left\{\frac{m}{d^k}:m,k \in \mathbb{Z}\right\},\] 
which is a commutative ring with the usual binary operations. We first recall the following straightforward: 

\medskip\noindent{\it Claim: }The units of the ring $\mathbb{Z}\left[\frac{1}{d}\right]$ are isomorphic to $\mathbb{Z}^l \times \mathbb{Z}_2$. In particular,
\[U\left(\mathbb{Z}\left[\frac{1}{d}\right]\right)=\left\{\pm\pr_1^{n_1}\hdots\pr_l^{n_l}: n_i \in \mathbb{Z}, 1 \le i \le l\right\}.\]

Let $u=\pm\pr_1^{n_1}\hdots\pr_l^{n_l}$ for $n_i \in \mathbb{Z}, 1 \le i \le l$. Let $k,\ti{k} \ge 1$ such that $q_i=km_i+n_i \ge 0$ and $\ti{q}_i=\ti{k}m_i-n_i \ge 0.$ Then 
\[u=\pm\frac{\pr_1^{q_1}\hdots\pr_l^{q_l}}{d^k} \text{ and }u^{-1}=\pm\frac{\pr_1^{\ti{q}_1}\hdots\pr_l^{\ti{q}_l}}{d^{\ti{k}}},\]
i.e., $u \in U\left(\mathbb{Z}\left[\frac{1}{d}\right]\right)$. Now suppose $u \in U\left(\mathbb{Z}\left[\frac{1}{d}\right]\right)$. Then there exist $m,m', k, k' \in \mathbb{Z}$ such that $u=\frac{m}{d^k}$ and $u^{-1}=\frac{m'}{d^{k'}}$. Thus, $mm'=d^{k+k'}$ or equivalently $m=\pm\pr_1^{q_1}\hdots\pr_l^{q_l}$ for positive integers $q_i \ge 0$. Hence, $u \in \{\pm\pr_1^{n_1}\hdots\pr_l^{n_l}: n_i \in \mathbb{Z}, 1 \le i \le l\}$. 

\medskip Let $\iso$ be a group isomorphism of $\mathbb{Z}\left[\frac{1}{d}\right]$. Since $\iso\left(\frac{m}{d^k}\right)=\frac{m}{d^k}\iso(1)$ for every $m, k \in \mathbb{Z}$ and $\iso(1).\iso^{-1}(1)=1$, $\iso(1)=u_{\iso}\in U\left(\mathbb{Z}\left[\frac{1}{d}\right]\right).$ Hence $\iso$ is an isomorphism, if and only if there exists a unit $u_{\iso} \in U\left(\mathbb{Z}\left[\frac{1}{d}\right]\right)$ such that $\iso(x)=u_{\iso}x$ for every $x \in \mathbb{Z}\left[\frac{1}{d}\right]$. For $\phi \in \Aut{U^+}$, let $\iso_\phi(x)=u_{\iso_\phi} x$ denote the isomorphism induced by $\phi$ on $\mathbb{Z}\left[\frac{1}{d}\right]$. Let
\[\tilde{I}=\left\{(n_1,\hdots,n_l) \in \mathbb{Z}^l: \text{ there exists }\phi \in \Aut{U^+} \text{ such that }u_{\iso_\phi}= \pm\pr_1^{n_1}\hdots\pr_l^{n_l}\right\}.\]
Thus $\tilde{I} \le \mathbb{Z}^l \le U\left(\mathbb{Z}\left[\frac{1}{d}\right]\right)$. Also $H \in \Aut{U^+}$, and $\iso_H(x)=dx$ for every $x \in \mathbb{Z}\left[\frac{1}{d}\right]$. Hence, $\mathbb{Z} \simeq \{(m_1k,\hdots,m_l k): k \in \mathbb{Z}\} \le \tilde{I}.$ Now, consider the map $h:\Aut{U^+} \to \tilde{I}$ such that $h(\phi)=(n_1,\hdots,n_l)$ where $u_{\iso_\phi}=\pm\pr_1^{n_1}\hdots\pr_l^{n_l}$, is a surjective group homomorphism with kernel $I_0=\{ \phi \in \Aut{U^+}: u_{\iso_\phi}=\pm 1\}.$ Suppose there exists $\phi \in I_0$ such that $u_{\iso_\phi}=-1$, then $\iso_\phi(\mathbb{Z})=\mathbb{Z}.$ In particular, $\phi$ lifts to an automorphism $\ti{\phi}$ of the covering space $\cover$. By the analysis in {\it Step 1} of the proof of Theorem \ref{t:escaping set},
\[\ti{\phi}(z,\z)=(\beta z+\gamma(\z), \alpha \z) \text{ and }\iso_{\ti{\phi}}(n)=n \text{ for every }n \in \mathbb{Z} \simeq \pi_1\left(\cover\right).\]
Let $\widehat{\Pi}$ be the covering map from $\cover$ to $U^+$, then $\widehat{\Pi}_*(n)=n$ for every $n \in \mathbb{Z} \simeq \uppi_1\left(\cover\right)$. But $\widehat{\Pi}_* \circ \iso_{\ti{\phi}}(n)=\iso_\phi \circ \widehat{\Pi}_*(n)$, i.e., $n=-n$ for every $n \in \mathbb{Z}$. This is not possible, hence $u_{\iso_\phi}=1$ for every $\phi \in I_0$ or $I_0={\rm Aut}_1(U^+)$ and ${\rm Aut}_1(U^+)\times \tilde{I}\simeq\Aut{U^+}$. Thus,
\[\C \times \mathbb{Z}\le {\rm Aut}_1(U^+)\times \mathbb{Z}\le {\rm Aut}_1(U^+)\times \tilde{I}\simeq\Aut{U^+} \le {\rm Aut}_1(U^+)\times  \mathbb{Z}^{l} \le \left(\mathbb{C}\rtimes \mathbb{Z}_{(d+d')(d-1)}\right)\times \mathbb{Z}^{l}.\]
This proves (ii). Now, if $d=\pr^m$ for a prime $\pr$ and positive integer $m \ge 1$, the above identity simplifies as
\[\C \times \mathbb{Z}\le \Aut{U^+} \le \left(\mathbb{C}\rtimes \mathbb{Z}_{(d+d')(d-1)}\right)\times \mathbb{Z}.\qedhere\]
\end{proof}
\begin{cor}\label{c:punctured short}
$\mathbb{C}\le {\rm Aut}(\Omega_c')\le \left(\mathbb{C}\rtimes \mathbb{Z}_{(d+d')(d-1)}\right) \times \mathbb{Z}^{l-1}$ where  $d=\pr_1^{m_1}\hdots\pr_l^{m_l}$ for primes $\pr_i, 1 \le i \le l$ and $m_i \ge 1$. 
\end{cor}
\begin{proof}
By the proof of Theorem \ref{t:revised BPV} (iii), for every $\al \in {\rm Aut}_1(\Omega_c')$ there exists a unique lift of the form $(\beta z+\gamma(\z), \alpha \z)$ such that $\alpha^{d+d'}=\beta$ and $\beta^{d-1}=\alpha^{d_0(d-1)}=1$, where $d_0$ is as in (\ref{d0}). Also $\uppi_1(\Omega_c')=\mathbb{Z}\left[\frac{1}{d}\right]$ and by repeating the same reasoning as in the proof of Theorem \ref{t:aut U+},
\[{\rm Aut}(\Omega_c') \simeq {\rm Aut}_1(\Omega_c') \times \ti{I}_c,\]
where $\tilde{I}_c=\left\{(n_1,\hdots,n_l) \in \mathbb{Z}^l: \text{ there exists }\phi \in \Aut{\Omega_c'} \text{ such that }u_{\iso_\phi}= \pm\pr_1^{n_1}\hdots\pr_l^{n_l}\right\}.$

\medskip\noindent{\it Claim: } $\tilde{I}_c$ is isomorphic to a subgroup $\mathbb{Z}^{l-1}$ for every $c>0$.

\medskip Suppose not, i.e., $\tilde{I}_c \simeq \mathbb{Z}^{l}$. Then there are integers $r_i \ge 1$ such that 
\[\tilde{I}_c=r_1 \mathbb{Z} \times \cdots \times r_l \mathbb{Z}.\]
Let $r=r_1r_2...r_l$, then $(rm_1,\hdots,rm_k) \in \tilde{I}_c$, which means that there exists $\psi \in \Aut{\Omega_c'}$ such that $u_{\iso_\psi}=d^{r}$. Define $\al=H^{-r} \circ \psi$. Then $\al(\Omega_c')=\Omega'_{cd^{-r}}$, and $\al$ is a biholomorphism that induces the identity map between $\uppi_1(\Omega_c')$ and $\uppi_1\left(\Omega_{cd^{-r}}'\right)$. Hence, $\al$ lifts to a biholomorphism between $\C \times \mathcal{A}_c$ to $\C \times \mathcal{A}_{cd^{-r}}$. Let 
\[A(z,\z)=\left({\rm A}_1(z,\z), {\rm A}_2(z,\z)\right)\]
be this biholomorphic lift. Now $z \mapsto {\rm A}_2(z,\z)$
is a bounded entire function, which implies that ${\rm A}_2(z,\z)\equiv h(\z)$ and thus
$A(z,\z)=( {\rm A}_1(z,\z),h(\z))$. Similarly, 
$A^{-1}(z,\z)=({\rm A}_1^*(z,\z),h^*(\z)).$ Since $A\circ A^{-1}(z,\z)=A^{-1}\circ A(z,\z))=(z,\z),$ it follows that $h\circ h^*\equiv h^*\circ h\equiv \sf{Id}$. Therefore, $h$ is a biholomorphism between $\mathcal{A}_c$ and $\mathcal{A}_{cd^{-r}}$,  which is not possible, unless $r=0$. Hence, the claim follows.

\medskip Thus, $\mathbb{C}\le {\rm Aut}_1(\Omega_c') \le {\rm Aut}(\Omega_c') \simeq {\rm Aut}_1(\Omega_c') \times \ti{I}_c \le (\mathbb{C}\rtimes \mathbb{Z}_{(d+d')(d-1)}) \times \mathbb{Z}^{l-1}.$
\end{proof}

\section{The subgroup \texorpdfstring{${\rm Aut}_1(U^+) \cap {\rm Aut}(\C^2)$ }{} and Proof of Theorem \ref{t:main}}
In this section, we will use the lifts of elements in ${\rm Aut}_1(U^+)$ to describe the elements of ${\rm Aut}_1(U^+) \cap {\rm Aut}(\C^2)$. The map $\G$, introduced in the proof of Theorem \ref{t:escaping set},  will continue to denote the isomorphism between ${\rm Aut}_1(U^+)$ and $\C \rtimes G$. The term $\G$-lift will be used to denote the lift of an automorphism of $U^+$ corresponding to the map $\G$, i.e., for $\al \in {\rm Aut}_1(U^+)$ if $\G(\al)=(c,\alpha)$ then the $\G$-lift of $\al$ is
\[A(z,\z)=(\alpha^{d+d'}z+\gamma(\z), \alpha \z) \text{ where }\gamma(\z_0)=c.\]
Furthermore, by {\it Step 3} of the proof of Theorem \ref{t:escaping set}, if $\G(\al)=(c,1)$ then the lift of $\al$ is $(z+c,\z)$. In particular, the lift is independent of the point $\z_0$ in $\C \rtimes \{1\}$.

\subsection{The group \texorpdfstring{$\G^{-1}\left(\C \rtimes \{1\}\right) \cap {\rm Aut}(\C^2)$}{}} The goal of this subsection is to prove Theorem \ref{t:Aut(U+) cap Aut(C2)} which says that $\G^{-1}\left(\C \rtimes \{1\}\right) \cap {\rm Aut}(\C^2)=\textsf{Identity}$. The proof will be presented through various cases. 
\begin{thm}\label{t:Aut(U+) cap Aut(C2)}
Let $\al \in {\rm Aut}_1(U^+) \cap {\rm Aut}(\C^2)$ be such that the $\G$-lift of $\al$ is of the form
\[A(z,\z)=(z+c_0,\z).\]
Then $c_0=0$, or equivalently $\al$ is the identity map.
\end{thm}
Define the group isomorphisms $\Phi_H^\pm: {\rm Aut}(U^+) \to {\rm Aut}(U^+)$
\begin{align}\label{e:group isomorphism}
\Phi_H^+(f)&=H\circ f \circ H^{-1} \text{ and }\Phi_H^-(f)=H^{-1}\circ f \circ H \text{ for } f \in {\rm Aut}(U^+) 
\end{align}
and note that from Proposition \ref{p1:ss:1-1}(i), ${\rm Aut}_1(U^+)$ is invariant under $\Phi_H^\pm$ with $\Phi_H^+\circ \Phi_H^-=\Phi_H^-\circ\Phi_H^+=\textsf{Identity}$. Recall from Theorem \ref{t:HOv-gen} that the lift of the map $H$ to $\cover$ is 
\[\w{H}(z,\z)=\left(\frac{a}{d}z+ Q(\z), \z^d\right),\]
where $Q$ is a polynomial of degree $d+d'$.
\begin{lem}\label{l:lift of Phi_H}
The $\G$-lifts of the automorphisms $\Phi_H^\pm(\al)$, denoted by $\w{\Phi_H^\pm(\al)}$, are of the form 
\[ \w{\Phi_H^+(\al)}(z,\z)= \left(z+\frac{a}{d}c_0, \z\right)  \text{ and }\w{\Phi_H^-(\al)}(z,\z)= \left(z+\frac{d}{a}c_0, \z\right),\]
where $\al \in \G^{-1}\left(\C \rtimes \{1\}\right)$.
\end{lem}
\begin{proof}
Note that $H\circ \al=\Phi_H^+(\al) \circ H$. Hence, there exists a lift of $H \circ \al$ to $\cover$ of the form
\[\w{H\circ \al}(z,\z)=\w{H} \circ A(z,\z)=\left(\frac{a}{d}z+ Q(\z)+\frac{a}{d}c_0, \z^d\right)=A' \circ \w{H}(z,\z),\]
where $A'(z,\z)=\left(z+\frac{a}{d}c_0, \z\right)$. Now by Theorem \ref{t:aut U+}, there exists $\al' \in {\rm Aut}_1(U^+)$ such that $A'$ corresponds to the $\G$-lift of $\al'$. By the above relation, both $H\circ \al$ and $\al' \circ H$ lift to the same map and hence, $\al'=H \circ \al \circ H^{-1}=\Phi_H^+(\al)$ and $A'$ corresponds to a lift of $\Phi_H^+(\al)$.

\medskip A similar argument works for the lift of $\Phi_H^-(\al)$.
\end{proof}
\subsection*{Proof of Theorem \ref{t:Aut(U+) cap Aut(C2)} when \texorpdfstring{$\frac{a}{d}$}{d/a} is a \texorpdfstring{$k-$}{k-}th root of unity for some \texorpdfstring{$k>0$}{k>0}} 
Note that
$$\w{H^k \circ \al \circ H^{-k}}(z,\z)=\w{{\Phi_H^+}^k(\al)}(z,\z)=(z+c_0,\z)=A(z,\z).$$
Hence $H^k \circ \al=\al \circ H^k$ and $H^k \circ \al^{-1}=\al^{-1} \circ H^k$ Now for $z \in K^-$, there exists $R_z \ge 0$ such that $\|H^{-n}(z)\| \le  R_z$. Since $\al(\ov{B(0;R_z)})$ is a compact subset of $\C^2$, there exists $M>0$ such that
\[G_H^-(\al(z))=d^{-n}G_H^-(\al \circ H^{-n}(z)) \le d^{-n}M\] 
for large $n$. Thus $\al(z) \in K^-$ and hence $\al(K^-)\subset K^-$. Similarly, by applying this to the automorphism $\al^{-1}$, we get $\al(K^-)=K^-$. Hence by Theorem 1.1 in \cite{BPV:rigidity}, $\al=L \circ H^{s'}$ where $L$ is an affine map of the form
\[L(x,y)=(e x+f, e' y+f'),\]
with $|e|=|e'|=1.$ Also $G_H^\pm\circ L(x,y)=G_H^\pm(x,y)$ and $L(K^\pm)=K^\pm$. For $(x,y) \in V^+ \cap L^{-1}(V^+)$ 
\[ |\phi(L(x,y))|=\exp({G_H^+(x,y)})=|\phi(x,y)|,\]
where $\phi$ is the B\"{o}ttcher function from Subsection \ref{subsection 1}. Since $V^+\cap L^{-1
}(V^+)$ is a non-empty open subset of $\C^2$ there exists $|\eta|=1$ such that
\[ \phi\circ L(x,y)=\eta \phi(x,y).\]
for every $(x,y) \in V_R^+.$ In particular, $L^* \omega=\omega$ where $\omega=\frac{d \phi}{\phi}$. Thus $\alpha(L(C))=\alpha(C)$ for every curve $C \in U^+$ or $L \in {\rm Aut}_1(U^+)$. Also, $\al \in {\rm Aut}_1(U^+)$ and hence from Proposition \ref{p:UedaBook}, $\al=L$ and $s'=0$. 

\medskip Further note that $\al \circ H^k =H^k \circ \al$, and hence $e'^{d}=e'$ and $e'^{d'}=e$. Thus there are finitely many choices for $e'$, and hence consequently, for $e$. For an integer $k \ge 1$, the locus of fixed points of $H^k$ forms a complex analytic subvariety that is compact since it is contained in both $K^{\pm}$, and hence $K$. By Bezout's theorem, there are at most $d^{2k}$ fixed points of $H^k$. Also, $\al$ should fix the set of fixed points of $H^k$. Hence there are finitely many choices of $f$ and $f'$ for a pair $(e,e')$ and this shows that there are only finitely many affine maps that can commute with $H^k$. 

\begin{rem}\label{r:commutative 1}
The steps above also imply that if $f \in {\rm Aut}(\C^2) \cap {\rm Aut}_1(U^+)$ is such that $f \circ H^m=H^m \circ f$ for some $m \in \mathbb{Z}$, then $f(K^\pm)=K^\pm$ and $f$ is an affine map. Conversely, if $f$ is an affine map preserving both $K^\pm$ then $f \in {\rm Aut}(\C^2) \cap {\rm Aut}_1(U^+)$.
\end{rem}

\medskip Now recall that the $\G$-lift of $\al$ is $A(z,\z)=(z+c_0, \z)$. Also, $\al^n$ are automorphisms of $\C^2$ that  $\G$-lift to $A^n(z,\z)=(z+nc_0, \z)$ for all $n \ge 1$. Since $\frac{d}{a}$ is a $k-$th root of unity by assumption, it follows that $H^k \circ \al^n=\al^n \circ H^k$. But there are only finitely many such maps that commute with $H^k$, and hence there exists $\tilde{n}>0$ such that $\al^{\tilde{n}}=\textsf{Identity}.$ In particular, as $\G$ is injective, the $\G$-lift of $\al^{\tilde{n}}$ is the identity. Hence
$A^{\tilde{n}}(z,\z)=(z+\tilde{n}c_0, \z)=(z,\z)$ or $c_0=0$. \qed

\medskip
To complete the proof in the other cases, i.e., when $\left|\frac{d}{a}\right|=1$ but is not a root of unity and $\left|\frac{d}{a}\right| \neq 1$, several intermediate results will be needed.

\begin{lem}\label{l:convergence in cover}
Let $\{\al_n\} \subset {\rm Aut}_1(U^+) \cap {\rm Aut}(\C^2)$ be such that for every $n \ge 1$, the $\G$-lift of $\al_n$ is of the form
\[A_n(z,\z)=(z+c_n, \z)\]
with $c_n \to c$ as $n \to \infty$. Then there exists $\al \in {\rm Aut}_1(U^+) \cap {\rm Aut}(\C^2)$ such that it $\G$-lifts to $\cover$ as  
\[A(z,\z)=(z+c, \z).\]
Further, both $\al_n \to \al$, $\al_n^{-1} \to \al^{-1}$ on compact subsets of $\C^2$.
\end{lem}
\begin{proof}
Recall that $\widehat{\Pi}$ is the covering map from $\cover$ to $U^+$ as in Theorem \ref{t:HOv-gen}. Thus for every $z \in U^+$, there exists $\eta_z>0$ such that $\widehat{\Pi}$ restricted to each component of $\widehat{\Pi}^{-1}(B(z, \eta_z))$ is a biholomorphism. Also, from {\it Step 4} in the proof of Theorem \ref{t:aut U+}, there exists $\al \in {\rm Aut}_1(U^+)$ such that $\G$-lift of $a$ is $A$.

\medskip Now fix a $z \in U^+$ and let $U$ be one such component of $\widehat{\Pi}^{-1}\big(B(\al(z),\eta_{\al(z)})\big)$ and $W=A^{-1}(U)$. There exists $\widetilde{W} \subset W$ such that $A_{n}(\widetilde{W}) \subset U$ for every $n \ge n_0$ and 
\[\widetilde{W} \cap \widehat{\Pi}^{-1}(z) \neq \emptyset.\]
Since $\widehat{\Pi}$ is continuous and the topologies on $\cover$ and $U^+$ are equivalent to the topology induced by the standard norm of $\C^2$, it follows that for a given $\ep>0$ there exists $\delta>0$ such that 
\[|\widehat{\Pi}(w)-\widehat{\Pi}(\tilde{w}|<\ep \text{ whenever }|w-\tilde{w}|<\delta,\] 
for every $w, w' \in U$. Furthermore, for $n \ge 1$ sufficiently large, note that $|A_{n}(w)-A(w)|<\delta$ for every $w \in \tilde{W}$, i.e., $|\widehat{\Pi} \circ A_{n}(w)-\widehat{\Pi} \circ A(w)|<\ep$. But $\widehat{\Pi} \circ A_{n}=\al_{n} \circ \widehat{\Pi}$ and $\widehat{\Pi} \circ A=\al \circ \widehat{\Pi}$. Thus
\[ |\al_n \circ \widehat{\Pi}(w)-\al \circ \widehat{\Pi}(w)|<\ep\]
for every $w \in \widetilde{W}.$ As $\widehat{\Pi}(\widetilde{W})$ is a neighbourhood of the point $z$, $\al_n \to \al$ locally near every $z \in U^+$ and consequently, on compact subsets of $U^+$. 

\medskip Similar arguments show that $\al_{n}^{-1} \to A^{-1}$ on compact subsets of $U^+$. Next, we prove the following claim, which will complete the proof.

\medskip{\it Claim}: The maps $\al^\pm$ extend as automorphisms of $\C^2$ and $\al_{n}^\pm \to \al^\pm$ on its compact subsets. 

\medskip Let $(x_0,y_0) \in \C^2$. Consider the complex line $L_{x_0}=\{(x_0,y): y \in \C\}$. Then $L_{x_0} \cap K^+$ is compact. Let $R_{x_0}=\max\{|x_0|+1,R_H\}>0$, where $R_H>0$ is the radius of filtration in (\ref{e:filtration_2}) and note that $L_{x_0} \cap K^+$ is compactly contained in the disc
\[ \mathcal{D}_{x_0}=\{(x_0,y):|y| <  R_{x_0}\}.\]
Let $D_{x_0}=L_{x_0}\setminus \mathcal{D}_{x_0}$ and for every $(x_0,y) \in D_{x_0}$ define
\begin{align}\label{e:extension}
\al_0(x_0,y)=\lim_{k\to \infty} \frac{1}{2 \pi i} \int_{\gamma_{x_0}} \frac{\al_{n}(x_0,w)}{w-y}dw=\frac{1}{2 \pi i} \int_{\gamma_{x_0}} \frac{\al(x_0,w)}{w-y}dw,
\end{align}
where $\gamma_{x_0}=(x_0,0)+(0,R_{x_0}e^{it})$, $t \in [0,2\pi]$. Since $\gamma_{x_0}$ is a compact subset of $U^+$ and the sequence of automorphisms $\al_{n} \to \al$, uniformly on compact subsets of $U^+$, $\al_0(x_0,y)$ is holomorphic in the $y$-variable on $D_{x_0}$. Also, $\al_0(x_0,y)=\al(x_0,y)$ on $\gamma_{x_0}$. Hence $\al(x_0,\cdot)$ extends as a holomorphic function on $L_{x_0}$. By repeating this reasoning for all $x$ sufficiently close to $x_0$ and noting that the same radius $R_{x_0}$ will work for such $x$, $\al$ extends to a function on $\C^2$ that is holomorphic in the $y$-variable. Furthermore, by (\ref{e:extension}), $\al_{n}$ converges uniformly on compact subsets of $V_{R_H} \cup V_{R_H}^-$ to $\al_0$. In particular, the sequence of automorphisms $\{\al_{n}\}$ converges uniformly on compact subsets of $\C^2$, and hence, $\al$ extends to a holomorphic map on $\C^2$. 

\medskip Finally, as $\al$ is a limit of the automorphisms $\al_{n}$ and $\al$ restricted to $U^+$ is injective, $\al$ is injective as well with $\al(\C^2) \subset \C^2$. Also, $\al(U^+)=U^+$ and $\al(K^+) \subset K^+$. Similar arguments (as above) applied to $\al^{-1}$, show that $\al^{-1}$ extends as an injective holomorphic map on $\C^2$ such that $\al^{-1}(U^+)=U^+$ and $\al^{-1}(K^+) \subset K^+$. Now $\al \circ \al^{-1}=\textsf{Identity}$ on $U^+$, hence on $\C^2$, i.e., $\al^\pm(\C^2)=\C^2$ and $\al^\pm(K^+)=K^+$.
\end{proof}

\begin{lem}\label{l:c neq 0}
Suppose $\frac{d}{a}$ is not a root of unity. Further, let $\al \in {\rm Aut}_1(U^+) \cap {\rm Aut}(\C^2)$ be such that the $\G$-lift of $\al$ is of the form $A(z,\z)=(z+c_0,\z) \text{ with } c_0 \neq 0.$ Then,
\begin{enumerate}
    \item [(i)] $\C \le {\rm Aut}_1(U^+) \cap {\rm Aut}(\C^2)$ if $\frac{a}{d} \in \C \setminus \mathbb{R}$.
     \item [(ii)] $\mathbb{R} \le {\rm Aut}_1(U^+) \cap {\rm Aut}(\C^2)$ if $\frac{a}{d} \in \mathbb{R}$.
\end{enumerate}
\end{lem}
\begin{proof}
For $n \ge 1$, define
\[c_n=\left(\frac{a}{d}\right)^n c_0 \text{ and } c_{-n}=\left(\frac{d}{a}\right)^n c_0.\]
Since the $\G$-lifts of $\al^l$, $l \in \mathbb{Z}$, are of the form
\[A^l(z,\z)=(z+lc_0, \z),\]
the $\G$-lifts of $\left(\Phi_H^+\right)^n\left(\al^l\right)$ and $\left(\Phi_H^-\right)^n\left(\al^l\right)$, $n \ge 1$, are of the form
\[A_n^l(z,\z)=(z+lc_n,\z) \text{ and } {\w{A}}_n^l(z,\z)=(z+lc_{-n},\z),\]
respectively. For $n \in \mathbb{Z}$, define
\[ S_n=\{kc_n+lc_{n+1}: \text{ where }k,l \in \mathbb{Z}\} \text{ and } S=\bigcup_{n=-\infty}^\infty S_n.\] 
Note that the $\G$-lift of $\left(\phi_H^+\right)^n(\al^k)\circ \left(\phi_H^+\right)^{n+1}(\al^l) \in {\rm Aut}_1(U^+) \cap {\rm Aut}(\C^2)$ is of the form
\[ A_{n}^k \circ A_{n+1}^l(z,\z)=(z+kc_n+lc_{n+1}, \z).\]
Similarly, the  $\G$-lift of $\left(\phi_H^-\right)^n(\al^k)\circ \left(\phi_H^-\right)^{n-1}(\al^l) \in {\rm Aut}_1(U^+) \cap {\rm Aut}(\C^2)$ is of the form
\[ \w{A}_{n}^k \circ \w{A}_{n-1}^l(z,\z)=(z+kc_{-n}+lc_{-n+1}, \z).\]
In particular, for every $\tilde{c} \in S$ there exists an element in $f \in {\rm Aut}_1(U^+) \cap {\rm Aut}(\C^2)$ such that the $\G$-lift of $f$ is $(z+\tilde{c},\z).$

\medskip
{\bf Case 1:} Let $\left| \frac{a}{d}\right| \neq 1$ and $\frac{a}{d} \in \mathbb{C} \setminus \mathbb{R}$. 

\medskip Note that $S_n$ is a lattice of points generated by the complex numbers (or the vectors) $c_n$ and $c_{n+1}$ for every integer $n \in \mathbb{Z}$. The minimum distance between any two distinct points in this lattice is given by
\[\min\{|p-q|: p,q \in S_n \text{ and }p \neq q\}\le \left(\left| \frac{a}{d}\right|^{n+1}+\left| \frac{a}{d}\right|^n\right)|c_0|.\]
If $\left| \frac{a}{d}\right|<1$, then for any point $p_0 \in \mathbb{R}^2$
\[\text{dist}(p_0,S)\le \text{dist}(p_0,S_n)<2\left|\frac{a}{d}\right|^n |c_0|\to 0 \text{ as }n \to \infty.\]
Otherwise, if $\left| \frac{a}{d}\right|>1$, by a similar argument as above, we see that 
\[\text{dist}(p_0,S)\le \text{dist}(p_0,S_{-n})<2\left|\frac{d}{a}\right|^{n-1} |c_0|\to 0 \text{ as }n \to \infty.\]
Thus, $S$ is a dense subset of $\mathbb{R}^2 \backsimeq \C$. Now, by Lemma \ref{l:convergence in cover}, it follows that $\C \le {\rm Aut}_1(U^+) \cap {\rm Aut}(\C^2)$.

\medskip{\bf Case 2:} Let $\left| \frac{a}{d}\right| \neq 1$ and $\frac{d}{a} \in \mathbb{R}$.
Then $S_n \subset \{ tc_0: t \in \mathbb{R}\}$ and 
\[\min\{|p-q|: p,q \in S_n \text{ and }p \neq q\}\le \left(\left| \frac{a}{d}\right|^{n+1}+\left| \frac{a}{d}\right|^n\right)|c_0|\]
or 
\[\min\{|p-q|: p,q \in S_{-n} \text{ and }p \neq q\}\le \left(\left| \frac{d}{a}\right|^{n-1}+\left| \frac{d}{a}\right|^n\right)|c_0|.\]
Hence by the same reasoning as in {\it Case 1}, $S$ is a dense subset $\{ tc_0: t \in \mathbb{R}\}.$ Thus by Lemma \ref{l:convergence in cover}, it follows that $\mathbb{R} \le {\rm Aut}_1(U^+) \cap {\rm Aut}(\C^2)$.

\medskip{\bf Case 3:} Let $\left| \frac{a}{d}\right| = 1$ and $\frac{a}{d} $ is not a root of unity.

\medskip Then for a positive integer $l\ge 1$, $\{lc_n\}_{n \in \mathbb{Z}}$ is a dense subset of the circle of radius $l|c_0|$. Hence, by Lemma \ref{l:convergence in cover}, it follows that for every $\ti{c}$ such that $|\ti{c}|=l|c_0|$ there exists an element in $f_{\ti{c}} \in {\rm Aut}_1(U^+) \cap {\rm Aut}(\C^2)$ such that the $\G$-lift of $f_{\ti{c}}$ is $(z+\tilde{c},\z).$ In particular, the $\G$-lift of $\al^l \circ f_{\ti{c}}$ is 
$$(z+lc_0+\ti{c}, \z).$$ 
Note that $0\le |lc_0+\ti{c}| \le 2l|c_0|$, with the bounds being attained at the points $\ti{c}=-lc_0$ and $\ti{c}=lc_0$. Also, $lc_0+\ti{c}$ is a continuous function on the circle of radius $l|c_0|$. Thus the $\G$-lift of $\left(\Phi_H^+\right)^n (\al^l \circ f_{\ti{c}}) \in {\rm Aut}_1(U^+) \cap {\rm Aut}(\C^2)$ is of the form
\[\Big(z+\left(\frac{a}{d}\right)^n (lc_0+\ti{c}), \z\Big).\]
Now the sequence $\left\{\left(\frac{a}{d}\right)^n (lc_0+\ti{c})\right\}$ is dense in the radius of circle $|lc_0+\ti{c}|$. Hence, again by Lemma \ref{l:convergence in cover}, it follows that for every $\eta$ such that $0 \le |\eta|=|lc_0+\ti{c}|\le 2l|c_0|$ there exists an element in $f_{\eta} \in {\rm Aut}_1(U^+) \cap {\rm Aut}(\C^2)$ such that the $\G$-lift of $f_{\eta}$ is $(z+\eta,\z).$ Since the above is true for every $l\ge 1$, it follows that for every $\eta \in \C$ there exists an element in $f_{\eta}$ such that the $\G$-lift of $f_{\eta}$ is $(z+\eta,\z).$ This completes the proof.
\end{proof}

\no Next, we complete the proof of Theorem \ref{t:Aut(U+) cap Aut(C2)} for the rest of the cases.

\subsection*{Proof of Theorem \ref{t:Aut(U+) cap Aut(C2)} when \texorpdfstring{$\frac{d}{a} \in \mathbb{C} $}{d/a in C -{0}} and not a root of unity} Suppose $c_0 \neq 0$ in the assumption of Theorem \ref{t:Aut(U+) cap Aut(C2)}. Then by Lemma \ref{l:c neq 0}, for every $r>0$ there exists $0<|\eta|<r $ and $f_\eta \in {\rm Aut}_1(U^+) \cap {\rm Aut}(\C^2)$ such that the $\G$-lift of $f_{\eta}$ is $\tilde{f}_\eta (z,\z)=(z+\eta,\z)$ on $\cover$. 

\medskip Recall the domains $\w{W}_M^+\subset U^+$ and $\w{\Om}_M^+ \subset \cover$ from the Subsection \ref{subsection 2} and the biholomorphism $\w{\Psi}$ between them. See also Figure \ref{f:cover} in this context. Note that for every $\ep>0$ there exists $\w{R}_\ep^1>0$ such that for $|y|>2\w{R}_\ep^1>0$, sufficiently large
\[ \w{\Psi}(0,y) \in D(0;\ep)\times D(0;\w{R}_\ep^1)^c.\]
By Lemma \ref{l:intermediate} for $\ep>0$ there exists $\w{R}_\ep^2>0$ such that
$$D(0;2\ep)\times D(0;\w{R}_\ep^2)^c \subset \w{\Om}_M^+.$$
Thus for a given $\ep>0$ there exists $\w{R}_\ep>\max\{\w{R}_\ep^1, \w{R}_\ep^2\}+\ep$ such that 
\[\left[\w{\Psi}\left( \{0\} \times D\left(0; \w{R}_\ep\right)^c\right) \right]_\ep\subset \w{\Om}_M^+,\]
where $[X]_\ep$ denote the $\ep$-neighborhood of a set $X$. Now, again choose $0<|\eta|<\ep$ such that $f_\eta \in {\rm Aut}_1(U^+) \cap {\rm Aut}(\C^2)$ and $\G$-lift of $f_{\eta}$ is $\tilde{f}_\eta (z,\z)=(z+\eta,\z)$. Thus for $|y|>2\w{R}_\ep^c$,
\begin{align}\label{e:lift relation 1}
f_\eta(0,y)=\w{\Psi}^{-1} \circ \tilde{f}_\eta \circ \w{\Psi}(0,y)  \in D(0;2\ep) \times D\left(0; \w{R}_\ep\right)^c \subset \w{W}_M^+.
\end{align}
Let $f_\eta(0,y)=(p_\eta(y), q_\eta(y))$. Since $f_\eta \in  {\rm Aut}(\C^2)$, both $p_\eta, q_\eta$ are entire functions on $\C$. Hence from (\ref{e:lift relation 1}), we have
\begin{align}\label{e:lift relation 2}
\w{\Psi} \circ (p_\eta(y),q_\eta(y))=\tilde{f}_\eta \circ \w{\Psi} (0,y).
\end{align}
Let us also recall the map $\w{\Psi}$ explicitly from subsection \ref{subsection 2}:
\begin{align}\label{e:map Psi}
\w{\Psi}(x,y)=\Big(\psi\big(x,\phi(x,y)\big)-R\big(\phi(x,y)\big), \phi(x,y) \Big), \text{ where }\psi(x,y)=y\int_0^x \frac{\partial \lambda}{\partial y}(t,y)dt,
\end{align}
where $\phi$ is the B\"{o}ttcher coordinate function defined in (\ref{e:Bottcher}) and $\lambda$, $R$ are the functions discussed in Lemma \ref{l:preliminary 1}, Remark \ref{r:preliminary 2} and Proposition \ref{p:preliminary 3} respectively. Then from (\ref{e:lift relation 2}), it follows that
\[\phi(p_\eta(y),q_\eta(y))=\phi(0,y).\]
\noindent By equating the second coordinate,
\begin{align*}
\psi(0,\phi(0,y))-R(\phi(0,y))+\eta&=\psi\big(p_\eta(y),\phi(p_\eta(y),q_\eta(y))\big)-R\left(\phi\big(p_\eta(y),q_\eta(y)\big)\right) \\
&=\psi\big(p_\eta(y),\phi(0,y)\big)-R(\phi(0,y)).
\end{align*}
Also note that $\psi(0,y)=0$ and hence the above reduces to
\[\eta=\psi\big(p_\eta(y),\phi(0,y)\big)=\phi(0,y)\int_0^{p_\eta(y)} \frac{\partial \lambda}{\partial y}\big(t,\phi(0,y)\big)dt.\]
By the estimates on $\phi$ and $\frac{\partial \lambda}{\partial y}(t,y)$ obtained in Lemma \ref{l:preliminary 1} and Remark \ref{r:preliminary 2}
\begin{align*}
\eta -yp_\eta(y)&=\phi(0,y)\int_0^{p_\eta(y)} \frac{\partial \lambda}{\partial y}\big(t,\phi(0,y)\big)dt-y \int_0^{p_\eta(y)} dt\\
&=\left(\frac{\phi(0,y)}{y}-1\right) y\int_0^{p_\eta(y)} \frac{\partial \lambda}{\partial y}\big(t,\phi(0,y)\big)dt+y  \int_0^{p_\eta(y)} \left(\frac{\partial \lambda}{\partial y}\big(t,\phi(0,y)\big)-1\right)dt.
\end{align*}
For given $\delta>0$ sufficiently small there exists  $\w{R}_\delta>\w{R}$, sufficiently large, such that $|y|>\w{R}_\delta$
\begin{align*}
    \left |\frac{\eta}{y}-p_\eta(y)\right| \le \delta(1+\delta)|p_\eta(y)|+\delta|p_\eta(y)|=(\delta^2+2 \delta)|p_\eta(y)|.
\end{align*}
Let $\delta=\frac{1}{4}$. Then for $|y|>\w{R}_{1/4}$,  $|p_\eta(y)|\le \frac{16}{7}\frac{ |\eta|}{|y|}.$ In particular, $p_\eta$ is an entire function and $p_\eta(y) \to 0$ as $y \to \infty$. Hence $p_\eta(y)\equiv 0$. Thus from the above, $\eta=\psi(0,\phi(0,y))=0$, which is a contradiction to the assumption. This establishes that $c_0=0$ and completes the proof of Theorem \ref{t:Aut(U+) cap Aut(C2)}. \qed

\begin{rem}
It should be noted that the implications of Lemma \ref{l:c neq 0} as well as Lemma \ref{l:convergence in cover} could have been avoided to complete the proof of Theorem \ref{t:Aut(U+) cap Aut(C2)}. However, understanding them provides additional information and helps to establish the contradiction.
\end{rem}

\subsection{The group \texorpdfstring{${\rm Aut}_1(U^+) \cap {\rm Aut}(\C^2)$}{}} The goal here is to further develop the observation in Theorem \ref{t:Aut(U+) cap Aut(C2)} to prove the following result. 
\begin{thm}\label{t:linear maps}
The group ${\rm Aut}_1(U^+) \cap {\rm Aut}(\C^2)$ is a subgroup of $\mathbb{Z}_{d_0(d-1)}$ where $d_0$ is as in (\ref{d0}). In particular, it is the set of all the affine maps that preserve both $K^\pm$.
\end{thm}
\begin{proof}
    Suppose not. Let the number of elements in ${\rm Aut}_1(U^+) \cap {\rm Aut}(\C^2)$ be more than $d_0(d-1)$. Recall the map $\G$ from {\it Step 5} in the proof of Theorem \ref{t:escaping set}. Then there exist distinct $f, g \in {\rm Aut}_1(U^+) \cap {\rm Aut}(\C^2)$ such that
    \[\G(f)=(c_1,\alpha) \text{ and }\G(g)=(c_2,\alpha)\]
    where $\alpha$ is a $d_0(d-1)$-root of unity. Thus, $\G(f \circ g^{-1})=(\tilde{c}, 1)$, or equivalently the $\G$-lift of $f \circ g^{-1}$ is of the form $(z+\tilde{c}, \z)$. Since $f \circ g^{-1} \in {\rm Aut}_1(U^+) \cap {\rm Aut}(\C^2)$, by Theorem \ref{t:Aut(U+) cap Aut(C2)}, $\tilde{c}=0$, i.e., $f \circ g^{-1}=\textsf{Identity}$ or $f=g$. This is a contradiction. 

    \medskip Now suppose $f \in {\rm Aut}_1(U^+) \cap {\rm Aut}(\C^2)$ then $\left(\Phi_H^+\right)^n(f) \in {\rm Aut}_1(U^+) \cap {\rm Aut}(\C^2)$ for every $n \ge 1$. Since there are only finitely many maps in ${\rm Aut}_1(U^+) \cap {\rm Aut}(\C^2)$, there exist distinct $n,m\ge 1$ such that
    \[ \left(\Phi_H^+\right)^n(f)=\left(\Phi_H^+\right)^m(f) \text{ or } H^{n-m} \circ f=f \circ H^{n-m}.\]
    Thus, from Remark \ref{r:commutative 1}, $f$ is an affine map that preserves both $K^\pm$. Also note $f^n \in {\rm Aut}_1(U^+) \cap {\rm Aut}(\C^2)$ for every $n \ge 1$. Thus, there exists $1 \le n_f \le d_0(d-1)$ such that $f^{n_f}=\textsf{Identity}.$
    
    \medskip Now let $L$ be an affine map preserving $K^+$. Then for every $n \ge 1$
    \[ L \circ H^{n}\circ L^{-1}\circ H^{-n} \in {\rm Aut}_1(U^+) \cap {\rm Aut}(\C^2). \]
    As before there exist distinct integers $m,n \ge 1$ such that
    \[ L \circ H^{n}\circ L^{-1}\circ H^{-n}=L \circ H^m\circ L^{-1}\circ H^{-m} \text{ or } H^{m-n} \circ L=L\circ H^{m-n}.\]
    Hence $L$ preserves both $K^\pm$ and the proof follows from the converse part of Remark \ref{r:commutative 1}.
\end{proof}
Now by exactly repeating the arguments in the later part of the above proof completes the 
\begin{proof}[Proof of Theorem \ref{t:main}]
Let $f \in {\rm Aut}(U^+) \cap {\rm Aut}(\C^2).$ Then, as before, for every $n \ge 1$
    \[ f \circ H^{n}\circ f^{-1}\circ H^{-n} \in {\rm Aut}_1(U^+) \cap {\rm Aut}(\C^2). \]
    In particular, there exist distinct integers $m,n \ge 1$ such that
    \[ f \circ H^{n}\circ f^{-1}\circ H^{-n}=f \circ H^m\circ f^{-1}\circ H^{-m} \text{ or } H^{m-n} \circ f=f\circ H^{m-n}.\]
    Thus, $f$ preserves both $K^\pm$. Hence by Theorem 1.1 in \cite{BPV:rigidity}, $f=L \circ \rho_H^{s'}$ where $L$ is an affine map that preserves both $K^\pm$ and $\rho_H$ is a H\'{e}non map preserving $K_H^\pm$ such that degree of $\rho_H$ (say $r'$) divides the degree of $f$. Also by Theorem 1.1 in \cite{BPV:rigidity}, $r'$ divides $d$. Thus there exist positive integer $r$ such that degree of $f^r$ is divisible by $d$. Further, there exits an integer $s$ such that 
    $L=f^r\circ H^{-s} \in {\rm Aut}_1(U^+) \cap {\rm Aut}(\C^2)$, i.e., by Theorem \ref{t:linear maps}, $L$ is an affine map. This completes the proof.
\end{proof}
We conclude this section with two immediate corollaries. First, the very rigid property of the set $K^+$ allows us to weaken the assumptions of Theorem \ref{t:main}, stated as
\begin{cor}
Let $H$ be a generalised H\'{e}non map of the form (\ref{e:henon}) and $\phi \in {\rm Aut}(\C^2)$ such that $\phi(K_H^+) \subset K_H^+$. Then $\phi^r=L \circ H^s$ where $L$ is a linear map preserving both $K_H^\pm$ and $r,s$ are integers.   
\end{cor}
\begin{proof}
Recall that $\mu_H^+=\frac{1}{2\pi} dd^c G_H^+$. Note that $\phi_*(\mu_H^+)$  is positive $dd^c$-closed $(1,1)$ current supported in $K_H^+$, hence by the uniqueness of closed positive currents supported in $K_H^\pm$\,---\,see \cite{FS}, \cite{DS:rigidity}\,---\,$\phi_*(\mu_H^+)=c \mu_H^+$ for some $c>0$. Thus, by the proof Theorem A of \cite{CD:rigidity}, $G_H^+\circ \phi=c G_H^+$. Hence $\phi(K_H^+)=K_H^+$, and the result follows from Theorem \ref{t:main}.
\end{proof}
Finally, we use the above to prove a version of \cite[Theorem A]{CD:rigidity} under comparatively weaker assumptions.
\begin{cor}\label{c:poly conjugate}
Let $H$ and $F$ be generalised H\'{e}non maps of the form (\ref{e:henon}) and assume that there exists an automorphism $\phi$ satisfying $\phi(K_F^+)\subset K_H^+$. Then $\phi$ is a polynomial automorphism and $H$ and $F$ are polynomially conjugated in $\C^2$.
\end{cor}
\begin{proof}
Note that, $\phi^{-1} \circ H \circ \phi(K_F^+)\subset K_F^+$ and is also an automorphism of $\C^2$. So by the above result, there exist an affine map $L$ and $r,s \in \mathbb{Z}$ with $r>0$ such that $\phi^{-1} \circ H^r \circ \phi=L \circ F^s$. In particular,
\[L^{-1} \circ \phi^{-1} \circ H^r \circ \phi \circ L=F^s \circ L, \text{ i.e., }(\phi \circ L)^{-1}\circ H^r \circ (\phi \circ L)=F^s \circ L.\]
Also, $L(K_F^\pm)=K_F^\pm$ and there exists $k \ge 1$ such that $F^k \circ L=L \circ F^k$.  Then there exist a subsequence $\{n_k\}$ of positive integers and a positive integer $n_0$ such that
\[(\phi \circ L)^{-1}\circ H^{rn_k} \circ (\phi \circ L)=F^{sn_k} \circ L^{n_0}.\]
Now for $z \in K_F^+ \setminus K_F$, the left hand side of the above identity is bounded, however the right side diverges if $s$ is a negative non-zero integer. Hence $s \ge 0$. By Theorem A of \cite{CD:rigidity}, it follows that $ \phi\circ L$ is a polynomial automorphism, and consequently $\phi$ is a polynomial.
\end{proof}

\section{Proof of Theorem \ref{t:equivalence of automorphisms} and Corollary \ref{c:automorphism of C2}}\label{s:4}
\noindent Recall that for $c >0$, 
\[ \Omega_c=\{z \in \C^2: G_H^+(z)<c\} \text{ and } \Om_c'=\Om_c \setminus K^+\]
Consider the group isomorphisms $\Phi_{H,c}^\pm: {\rm Aut}(\Omega_c') \to {\rm Aut}(\Omega_{cd^{\pm}}')$ on ${\rm Aut}(\Omega_c')$, analogous to the definitions in (\ref{e:group isomorphism}).
\begin{align}\label{e:group isomorphisms_c}
\Phi_{H,c}^+(f_c)&=H\circ f_c \circ H^{-1} \text{ and }\Phi_{H,c}^-(f_c)=H^{-1}\circ f_c \circ H\text{ for } f_c \in {\rm Aut}(\Omega_c').
\end{align}
for every $c>0$. As before,  
\[ \Phi_{H,c}^\pm\left({\rm Aut}_1(\Omega_c')\right)={\rm Aut}_1(\Omega_{cd^\pm}').\] 
\begin{proof}[Proof of Theorem \ref{t:equivalence of automorphisms}]
\medskip For $\al \in  {\rm Aut}_1(U^+)$, by Theorem \ref{t:escaping set}, there exists a unique $\G$-lift of the form $\tilde{\al}=(\beta z+\gamma(\z), \alpha \z)$, such that $\alpha^{d_0(d-1)}=1$ and $\beta=\alpha^{d+d'}$. Note $\tilde{\al}(\C \times \mathcal{A}_c)=\C \times \mathcal{A}_c$ for every $c>0$. Hence, by Theorem \ref{t:revised BPV}, $\al(\Omega_c')=\Omega_c'$ for every $c>0$ and $\al$ induces identity in the fundamental group of $\Omega_c'$, i.e., $\al \in {\rm Aut}_1(\Omega_c')$ and ${\rm Aut}_1(U^+) \subset {\rm Aut}_1(\Omega_c')$ for every $c>0$. 

\medskip Recall the map $\G_c$ discussed in the proof Theorem \ref{t:revised BPV}, depending on the unique lift of form (\ref{e:unique lift}) of every element in ${\rm Aut}_1(\Omega_c')$. As in earlier sections, we will denote this lift as the $\G_c$-lift. By Theorem \ref{t:revised BPV}, the covering map $\h{\Pi}$ is the same on every $\C \times \mathcal{A}_c$, where $c>0$. Suppose $\al \in {\rm Aut}_1(\Omega_c')$ is such that $\G_c$-lift of $\al$ is $(z+\eta, \z)$. Now by Remark \ref{r:short c2 and Gc}, $(z+\eta,\z)$ also corresponds to the $\G$-lift of an element $\al' \in {\rm Aut}_1(U^+)$. Thus, by Theorem \ref{t:revised BPV}, on $\C \times \mathcal{A}_{c}$
\[\widehat{\Pi} \circ (z+\eta, \z)=\al \circ \widehat{\Pi}(z, \z), \text{  and }\widehat{\Pi} \circ ( z+\eta,  \z)=\al' \circ \widehat{\Pi}(z, \z)\]
on $\cover$. In particular, $\al'$ is the extension of the map $\al$ from $\Omega_c'$ to $U^+$. 

\medskip Suppose $\al \in {\rm Aut}_1(\Omega_c')$ is such that $\G_c$-lift of $\al$ is 
$$(\alpha^{d+d'}z+\gamma(\z), \alpha \z)$$
where $\alpha \in \mathbb{Z}_{d_0(d-1)}$ and $\alpha \neq 1$. Let $\al_n=(\Phi_{H,c}^+)^n \in {\rm Aut}_1(\Omega'_{cd^n})$ for every $n \ge 1$. Furthermore, let the $\G_{cd^n}$-lift of $\al_n$ be of the form $(a_n,\alpha_n)$. Then for some positive integers $m>n\ge 1$, $\alpha_n=\alpha_m$. In particular as $m>n$, $\al_n \circ \al_m^{-1} \in {\rm Aut}_1(\Omega'_{cd^n})$ and the $\G_{cd^n}$-lift of $\al_n \circ \al_m^{-1}$ is of the form $(z+\eta_{nm}, \z)$, $\eta_{nm} \in \C$, by similar steps as in {\it Step 4} of Theorem \ref{t:escaping set} and Remark \ref{r:short c2 and Gc}. Now, by the above, $f=\al_n \circ \al_m^{-1}$ extends to $U^+$ and is in ${\rm Aut}_1(U^+)$. Consequently, $f \in {\rm Aut}_1(\Omega_{cd^m}')$ by the first part of this proof. Hence,
\[\al_n=f \circ \al_m\in {\rm Aut}_1(\Omega_{cd^m}').\] 
Thus $\al_n$ extends from $\Omega_{cd^n}'$ to $\Omega_{cd^m}'$, or equivalently $\al=(\Phi_H^-)(\al_n)$ extends from $\Omega_{c}'$ to $\Omega_{cd^{m-n}}'$.

\medskip Repeating the above argument inductively, we get that $\al$ extends to an automorphism of $\Omega_{cd^{k(m-n)}}'$ for every $k \ge 1$. Now $U^+=\cup_{k=1}^\infty \Omega_{cd^{k(m-n)}}'$. Thus, $\al $ extends to an element of ${\rm Aut}_1(U^+)$ or ${\rm Aut}_1(\Omega_c')\subset {\rm Aut}_1(U^+).$
\end{proof}

\noindent Next as an application of Theorem \ref{t:equivalence of automorphisms}, we will complete 
\begin{proof}[Proof of Corollary \ref{c:automorphism of C2}] Note that by \cite[Proposition 1.2]{BPV:IMRN}, $f(K^+)=K^+$, $f(\Omega_b)=\Omega_b$ for $0<b<c$ and $f \in {\rm Aut}(\Omega_c')$. For every $n \ge 1$ let
\[\tilde{f}_n=f \circ H^n\circ f^{-1} \circ H^{-n} \text{ on } \Omega_c.\]
Then $\tilde{f}_n \in {\rm Aut}_1(\Omega_c')$. By Theorem \ref{t:equivalence of automorphisms}, $\tilde{f}_n$ extends to an automorphism of $U^+$ that induces identity on the fundamental group of $U^+$. Also, as $\ti{f}_n$ is holomorphic on a neighbourhood of $K^+$, $\ti{f}_n$ extends as an automorphism of $\C^2$, say $F_n$. Thus, by Theorem \ref{t:Aut(U+) cap Aut(C2)}, $F_n=L_n$ where $L_n$ is an affine map preserving $K^+$. By Theorem \ref{t:main}, there are only finitely many choices for every $L_n$, $n \ge 1$. Also, $f$ extends to an automorphism of $\Omega_{cd^n}$ as
\[f^{-1}=H^n \circ f^{-1}\circ H^{-n} \circ L_n^{-1} \text{ on } \Omega_{cd^n}.\]
Thus $f \in {\rm Aut}(\C^2)$ as $\C^2=\cup_n\Omega_{cd^n}$, with $f(K^+)=K^+$. Thus by Theorem \ref{t:main}, $f^r=L \circ H^s$ where $r,s \in \mathbb{Z}$ where $L$ is an affine map preserving $K^\pm$. Further $f^r(\Omega_c)=\Omega_c$, hence $f^r=L$ or $f \in {\rm Aut}_1(U^+) \cap {\rm Aut}(\C^2)$. Now the proof follows from Theorem \ref{t:linear maps}.
\end{proof}

Finally, we conclude with the following applications on biholomorphic {\it Short} $\C^2$'s
\begin{thm}\label{t:equivalence of short c2} 
Let $H$ and $F$ be generalised H\'{e}non maps of the form (\ref{e:henon}) of the same degree $d \ge 2$. Suppose that $\Omega_{H,c}$ and $\Omega_{F,c'}$ are biholomorphic, where
\[\Omega_{H,c}=\{z \in \C^2: G_H^+(z)<c\} \text{ and }\Omega_{F,c'}=\{z \in \C^2: G_F^+(z)<c'\}, \]
for some $c, c'>0$. Then there exist polynomial maps $P_1$ and $P_2$  such that 
\[H=P_1 \circ F \circ P_2.\]
In particular, if $d=\pr$ for some prime $\pr\ge 2$ then $P_1$ and $P_2$ are affine maps.
\end{thm}
\begin{proof}  
Let $K_H^+$ and $K_F^+$ denote the non-escaping set of $F$ and $H$, respectively. Let $p$ be a saddle point of $H$ such that the stable manifold $W^s(p)$ is dense in $J_H^+$. Since $J_H^+=\partial K_H^+$, $J_F^+=\partial K_F^+$ are non-smooth (see \cite{BS5}), $G_F^+\circ \phi(W^s_p)$ is bounded and subharmonic, i.e., it is constant. Hence, $\phi(J_H^+) \subset K_F^+$. Similarly, $\phi^{-1}(J_F^+) \subset K_H^+$. Now by following the arguments as in the proof of Proposition 1.2 in \cite{BPV:IMRN}, $G_F^+ \circ \phi(\partial \Omega_b)=a_b>0$ for every $0<b<c$. Hence the punctured short $\C^2$'s are biholomorphic, i.e.,
\[\Omega_{H,c} \setminus K_H^+=\Omega_{H,c}'\simeq  \Omega_{F,{c'}}'=\Omega_{F,{c'}} \setminus K_F^+.\]
The fundamental groups of the punctured {\it Short} $\C^2$'s are isomorphic, in particular,
\[\mathbb{Z}\left[\frac{1}{d}\right]\simeq\uppi_1 \left(\Omega_{H,c}'\right)\simeq\uppi_1 \left(\Omega_{F,{c'}}'\right).\]
Let $\phi: \Omega_{F,c'}\to \Omega_{H,c}$ be the biholomorphism. Define 
\begin{align}\label{e:bihol}
\al_n=\phi^{-1} \circ F^n \circ \phi \circ H^{-n},
\end{align}
for every $n \ge 1$. Note that each $\al_n$ induces the identity map on the fundamental group of $\Omega_{H,cd^n}$, i.e., $\al_n \in {\rm Aut}_1(\Omega_{H,cd^n}')$ for every $n \ge 1$. By Theorem \ref{t:equivalence of automorphisms}, every $\al_n$ extends to an automorphism of $U_H^+$, that induces identity on the fundamental group of $U_H^+$. Hence $\al_n \in {\rm Aut}_1(U_H^+) \cap \Aut{\C^2}$ and $\al_n(K_H^+)=K_H^+$. Further, the above identity (\ref{e:bihol}), now can be used to extend $\phi$ from $\Omega_{F,d^nc'}$ to $\Omega_{H,cd^n}$, via the relation
\[\phi= F^n \circ \phi \circ H^{-n}\circ \al_n^{-1}.\]
Thus $\phi$ extends to an automorphism of $\C^2$ such that $\phi(K_F^+)=K_H^+$, or equivalently their corresponding escaping sets $U_F^+$ and $U_H^+$ are biholomorphic, via the map $\phi$. In particular by Theorem \ref{t:main}, each $\al_n$, $n \ge 1$ is an affine map such that $\al_n(K_H^-)=K_H^-$, $G_H^- \circ {\al}_n=G_H^-$ and $\al_n \in \mathcal{L}$. Since there are only finitely many choice for every $\al_n$, it follows from (\ref{e:bihol}) that $\phi(K_H^-)=K_F^-$ or $\phi(U_H^-)=U_F^-$. Thus the statement follows from Corollary \ref{c:poly conjugate}. 

\medskip  Now $\phi^*(\mu_H^-)$ and $\phi^*(\mu_H^+)$ are $dd^c$-closed $(1,1)$-positive currents supported on $K_H^-$ and $K_H^-$, respectively. By arguing similarly as in the proof of Theorem A in \cite[page 3748]{CD:rigidity}, there exist constants $a,b>0$ such that
\[G_F^+\circ \phi(z)=a G_H^+(z) \text{ and }G_F^-\circ \phi(z)=b G_H^-(z). \] Since $\phi$ is an automorphism of $\C^2$,
\[1=\int_{\C^2} \mu_H^+ \wedge \mu_H^-=\int_{\phi(\C^2)} \phi^*(\mu_H^+ \wedge \mu_H^-)=\int_{\C^2} ab\mu_F^+ \wedge \mu_F^-=ab.\]
Hence $b=a^{-1}$. Assuming $d=\pr$ we prove the following

\medskip
{\it Claim: }There exists $s \in \mathbb{Z}$ such that 
\[G_F^+\circ \phi\circ H^s(z)= G_H^+(z) \text{ and }G_F^-\circ \phi\circ H^s(z)= G_H^-(z). \]
Note that $\phi$ induces an isomorphism between the fundamental groups of $U_H^+$ and $U_F^+$. In particular, $\iso_\phi(x)=\pr^{-s} x$ for some $s \in \mathbb{Z}$, since $d=\pr$. See proof of Theorem \ref{t:aut U+}, in this context. Thus $f=\phi \circ H^s$ induces identity on $\uppi_1(U^+_H)$. Hence, by the arguments as in {\it Step 1} of the proof of Theorem \ref{t:escaping set}, $f$ lifts as an automorphism of $\cover$, say $\tilde{f}$ such that
\[ f\circ \h{\Pi}_H=\h{\Pi}_F \circ \tilde{f}. \]
Again by the arguments as in {\it Step 1} of the proof of Theorem \ref{t:escaping set},
\[ \ti{f}(z,\z)=(\beta(\z)(z)+\gamma(\z),\alpha \z )\]
where $|\alpha|=1$ and $\beta$, $\gamma$ holomorphic function on $\C \setminus \ov{\mathbb{D}}$. Thus if $(x,y) \in \Omega_{c,H}$. Let $(z,\z) \in \cover$ such that $\h{\Pi}_H(z,\z)=(x,y)$. Then $\h{\Pi}_F \circ \tilde{f}(z,\z) \in \Omega_{c,F}'$, or $f(\Omega_{c,H}')\subset \Omega_{c,F}'$. Similarly for $f^{-1}$, hence $f(\Omega_{c,H}')=\Omega_{c,F}'$. Thus the claim.

\medskip Thus $G_F \circ f=G_H$ where $G_H=\max\{G_H^+,G_H^-\}$ and $G_F=\max\{G_F^+,G_F^-\}$. Now for $p \in \C^2$
\[ G_F(p)=\log|p|+O(1) \text{ and } G_H(f(p))=\log|f(p)|+O(1)=\log|p|
+O(1).\]
Hence $\|f(p)\|<M\|p\|+C$ for every $p \in \C$. This proves that $f$ is an affine map. Similarly $g=F^{-1} \circ f \circ H$ is also an affine map as it induces identity on the fundamental group. Thus
\[F=f \circ H \circ g.\qedhere\]
\end{proof}
\begin{cor}\label{c:equivalence of short c2}
Let $F$ be generalised H\'{e}non maps of the form (\ref{e:henon}) and $H$ be a simple H\'{e}non map of degree $d_H=\pr$ for some prime $\pr\ge 2$. Suppose that $\Omega_{H,c}$ and $\Omega_{F,c'}$ are biholomorphic then there exist $m \ge 1$ and affine maps $A_1$ and $A_2$ such that 
\[H^m=A_1 \circ F \circ A_2.\]
\end{cor}
\begin{proof} 
As $\Omega_{H,c}$ and $\Omega_{F,c'}$ are biholomorphic,
\[\mathbb{Z}\left[\frac{1}{\pr}\right]\simeq\uppi_1 \left(\Omega_{H,c}'\right)\simeq\uppi_1 \left(\Omega_{F,{c'}}'\right).\]
Hence $d_F=\pr^m$ for some $m \ge 1$ and the degrees of $F$ and $H^m$ are same. Let $H'= H^m$, then $G_H=G_{H'}$. Now by replicating the above proof with $H=H'$ (except, the map $f=\phi \circ H^s$) proves the statement.
\end{proof}
\begin{rem}
The above Corollary \ref{c:equivalence of short c2} proves that there exists a positive integer $m \ge 1$ such that $m-$the root of an appropriate linear conjugate of $F$, which is known to be a polynomial by \cite{BF:roots} is $H$, provided they induce biholomorphic short $\C^2$'s at certain values.

\medskip
Finally also note, the isomorphism induced by the map $\phi$ (in the above proofs) on the fundamental group $\mathbb{Z}\left[\frac{1}{d}\right]$, is $\iso_\phi(x)=u_\phi x$,  where $u_\phi$ is an unit in the ring $\mathbb{Z}\left[\frac{1}{d}\right]$. Suppose $u_\phi d^n  \neq 1$ for every $n \in \mathbb{Z}$, then the above argument will not work, hence the above results (on affine conjugation) could not be stated generally for arbitrary degrees $d \ge 2$.
\end{rem}


{\small\bibliographystyle{amsplain}

}
\end{document}